\documentclass[11pt]{article}

\usepackage{amsfonts}
\usepackage{amssymb}
\usepackage{amsmath,mathrsfs}
\usepackage{amsthm}
\usepackage{latexsym}
\usepackage{layout}
\usepackage{units}
\usepackage{relsize}
\usepackage{tikz}
\usetikzlibrary{shapes,arrows}

\usepackage{enumerate}
\usepackage{enumitem}
\usepackage{textcomp}
\usepackage{amsbsy}
\usepackage{latexsym}
\usepackage[mathscr]{eucal}
\usepackage{amsfonts}
\usepackage{amssymb}
\usepackage{epsfig} 
\usepackage{epstopdf}
\usepackage{framed}

\usepackage[english]{babel}

\theoremstyle{plain}
\begingroup
\newtheorem{theorem}{Theorem}[section]
\newtheorem{lemma}[theorem]{Lemma}
\newtheorem{proposition}[theorem]{Proposition}
\newtheorem{corollary}[theorem]{Corollary}
\endgroup

\theoremstyle{definition}
\begingroup
\newtheorem{definition}[theorem]{Definition}
\newtheorem{problem}{Problem}
\newtheorem{remark}[theorem]{Remark}

\endgroup
 
\theoremstyle{remark}
\begingroup

\endgroup

\mathchardef\emptyset="001F

\numberwithin{equation}{section}

\newcommand{\R}{{\mathbb R}}
\newcommand{\N}{{\mathbb N}}
\newcommand{\PP}{{\mathcal P}_1}
\newcommand{\Pc}{{\mathcal P}_c}

\newcommand{\U}{{\mathcal U}}

\newcommand{\supp}{\mathrm{supp}}

\newcommand{\Pt}[1]{\left( #1 \right)}

\newcommand{\Pa}[1]{\langle #1 \rangle}

\renewcommand{\r}[1]{\eqref{#1}}
\renewcommand{\cal}[1]{\mathcal{#1}}

\newcommand{\bi}{\begin{itemize}}
\newcommand{\ei}{\end{itemize}}
\newcommand{\be}{\begin{enumerate}}
\newcommand{\ee}{\end{enumerate}}

\newcommand{\bd}{\begin{description}}
\newcommand{\ed}{\end{description}}
\renewcommand{\i}{\item}

\newcommand{\bqn}{\begin{eqnarray}}
\newcommand{\eqn}{\end{eqnarray}}
\newcommand{\eqnn}{\nonumber\end{eqnarray}}
\newcommand{\eqnl}[1]{\label{#1}\end{eqnarray}}

\newcommand{\nn}{\nonumber\\}

\newcommand{\ba}[1]{\begin{array}{#1}}
\newcommand{\ea}{\end{array}}

\newcommand{\bproof}{\begin{proof}}
\newcommand{\eproof}{\end{proof}}

\newcommand{\bl}{\begin{lemma}}
\newcommand{\el}{\end{lemma}}
\newcommand{\bp}{\begin{proposition}}
\newcommand{\ep}{\end{proposition}}
\newcommand{\bc}{\begin{corollary}}
\newcommand{\ec}{\end{corollary}}
\newcommand{\bdeff}{\begin{definition}}
\newcommand{\edeff}{\end{definition}}
\newcommand{\brem}{\begin{remark}\rm}
\newcommand{\erem}{\end{remark}}

\newcommand{\eps}{\varepsilon}

\newcommand{\om}{\omega}

\newcommand{\Id}{\mathrm{Id}}

\newcommand{\weak}{\rightharpoonup}
\renewcommand{\H}{\mathbb{H}}

\newcommand{\Y}{{\mathcal Y}}
\newcommand{\W}{{\mathcal W}}
\newcommand{\Lip}{{\mathrm{Lip}}}
\newcommand{\Jac}{{\mathbf{D}}}

\begin{document}

\title{Mean-Field Pontryagin Maximum Principle}


\author{Mattia Bongini\thanks{Technische Universit\"at M\"unchen, Fakult\"at Mathematik, Boltzmannstrasse 3
 D-85748, Garching bei M\"unchen, Germany. {\tt mattia.bongini@ma.tum.de}}, Massimo Fornasier\thanks{Technische Universit\"at M\"unchen, Fakult\"at Mathematik, Boltzmannstrasse 3
 D-85748, Garching bei M\"unchen, Germany. {\tt massimo.fornasier@ma.tum.de}}, Francesco Rossi\thanks{Aix Marseille Universit\'e, CNRS, ENSAM, Universit\'e de Toulon, LSIS UMR 7296, 13397, Marseille, France. {\tt francesco.rossi@lsis.org}}, Francesco Solombrino\thanks{Technische Universit\"at M\"unchen, Fakult\"at Mathematik, Boltzmannstrasse 3
 D-85748, Garching bei M\"unchen, Germany. {\tt francesco.solombrino@ma.tum.de}}}

\maketitle

\begin{abstract}
We derive a Maximum Principle for optimal control problems with constraints given by the coupling of a system of ODEs and a PDE of Vlasov-type. Such problems arise naturally as $\Gamma$-limits of optimal control problems subject to ODE constraints, modeling, for instance, external interventions on crowd dynamics. We obtain these first-order optimality conditions in the form of Hamiltonian flows in the Wasserstein space of probability measures with forward-backward boundary conditions with respect to the first and second marginals, respectively. In particular, we recover the equations and their solutions by means of a constructive procedure, which can be seen as the mean-field limit of the Pontryagin Maximum Principle applied to the discrete optimal control problems, under a suitable scaling of the adjoint variables. 
\end{abstract}




\section{Introduction}\label{s:intro}

The study of large crowds of interacting agents has received a significantly growing attention in the mathematical literature of the last decade with applications in biology, ecology,  social sciences, and economics. Starting from the seminal papers \cite{helbing2000simulating,hughes2002continuum,reynolds1987flocks,vicsek1995novel}, emphasis has been put on self-organization, i.e., the formation of macroscopic patterns from the superimposition of simple, reiterated, binary interactions.  A quintessential situation is the convergence of a crowd to a common state, which may be called consensus, agreement, or rendezvous. Several examples show that spontaneous convergence to pattern formation is not always guaranteed, e.g., for highly dispersed initial configurations in consensus problems \cite{cucker2011general,CS,d2006self,motsch2014heterophilious}, hence, the issue of controlling and stabilizing these systems arises naturally.
Two major interpretations of control of multiagent systems have received much attention: on the one hand, with the \emph{decentralized approach}, the problem is recast into a game-theoretic framework, where agents optimize their individual cost and solutions correspond to Nash equilibria.  On the other hand, following the concept of \emph{centralized intervention}, an external policy-maker controlling the dynamics is introduced, with the task of minimizing its intervention cost. 

When dealing with large populations, in both cases one faces the well-known problem of the \emph{curse of dimensionality}, term first coined by Bellman precisely in the context of dynamic optimization: the complexity of numerical computations of the solutions of the above problems blows up as the size of the population increases. A possible way out is the so-called \emph{mean-field approach}, where the individual influence of the entire population on the dynamics of a single agent is replaced by an averaged one. This substitution principle results in a unique mean-field equation and allows the computation of solutions, cutting loose from the dimensionality.

In the game-theoretic setting, the mean-field approach has led to the development of \emph{mean-field games} \cite{huang2003individual,lasry2007mean}, which model populations whose agents are competing \emph{freely} with the others towards the maximization of their individual payoff, as for instance in the financial market. The landmark feature of such systems is their capability to autonomously stabilize without external intervention. However, in reality, societies exhibit either convergence to undesired patterns or tendencies toward instability that only an external government can successfully dominate. The need of such interventions, together with the limited amount of resources that governments have at their disposal, makes the design of parsimonious stabilization strategies a key issue, which has been extensively studied in the context of dynamics given by systems of ODEs, see \cite{bofo13,bongini2014sparse,bonginijunge2014sparse,bongini2015conditional,caponigro2015sparse}.

Nevertheless, the concept of sparse control has to be handled with care when trying to generalize it at the level of a mean-field dynamics. Indeed, the indistinguishability of agents is a fundamental property of the mean-field setting, and it is in sharp contrast with controls acting sparsely on specific agents. Figuratively, trying to stabilize a huge crowds with these controls is like steering a river by means of toothpicks! A first solution to this ambiguity was given in \cite{bensoussan2013mean,MFOC}, where the control is defined as a locally Lipschitz feedback control with respect to the state variables, and sparsity refers to its property of having a small support. Such concept was successfully used in \cite{piccoli2014control} to implement sparse stabilizers for a consensus problem. This interpretation of sparsity appears also in the framework of the control of more classical PDEs, see \cite{clason2011duality,pieper2013priori,privat2013optimal,stadler2009elliptic}. An 
alternative solution for a proper definition of sparse mean-field control was proposed in \cite{fornasier2014mean}, where the control is sparsely applied on a finite number of individuals immersed in the mean-field dynamics of the rest of the population, resulting in a system where the controlled ODEs are coupled with a control-free mean-field PDE (but indirectly controlled via the coupling). The same kind of control was considered in \cite{albi2015invisible} to model the efficient evacuation of a large crowd of pedestrians with the help of very few informed agents.

While in the context of mean-field games and optimal control problems with PDE constraints, first-order optimality conditions have received enormous attention, see for instance \cite{andersson2011maximum,burger2013mean,carmona2013control,raymond1999hamiltonian}, up to now no corresponding results have appeared in the literature for coupled ODE-PDE systems of the kind considered in \cite{fornasier2014mean}, to the best of our knowledge. This paper is devoted to the development of a Pontryagin Maximum Principle to characterize optima of such control problems. We first remark that we are not interested in all possible optima, but mainly on those which arise as limits of optimal strategies of the original discrete problems. We call this subclass of the set of optima \emph{mean-field optimal controls} (see Definition \ref{d-mfoc}). The interest in this class complies with the wish of using the continuous models as approximations of the finite-dimensional ones. 
Furthermore in the model cases considered in \cite{fornasier2014mean,MFOC}, it is exactly the existence of mean-field optimal controls that is proved.

We summarize our result, borrowing a leaf from the diagram in \cite{carmona2013control}, as follows:

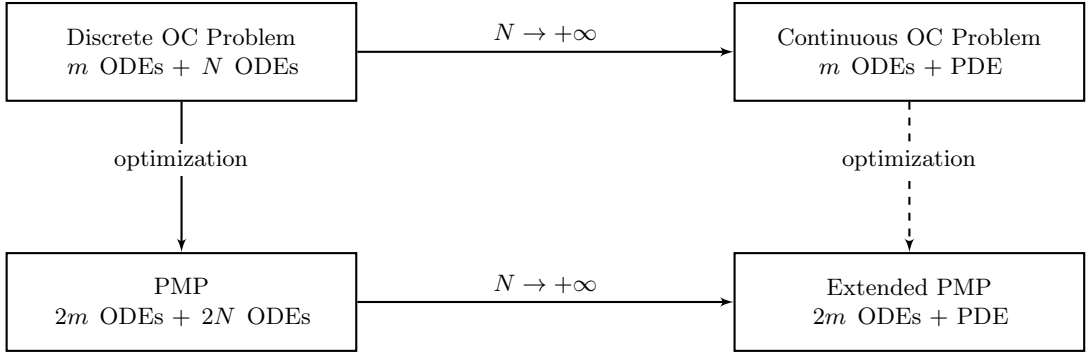
\begin{center}
  \footnotesize
  \begin{tikzpicture}[auto,
    block_center/.style ={rectangle, draw=black, thick, fill=white,
      text width=8em, text centered,
      minimum height=4em},
    block_left/.style ={rectangle, draw=black, thick, fill=white,
      text width=13em,  text centered, minimum height=4em, inner sep=6pt},
    block_noborder/.style ={rectangle, draw=none, thick, fill=none,
      text width=18em, text centered, minimum height=1em},
    block_assign/.style ={rectangle, draw=black, thick, fill=white,
      text width=18em, text ragged, minimum height=3em, inner sep=6pt},
    block_lost/.style ={rectangle, draw=black, thick, fill=white,
      text width=16em, text ragged, minimum height=3em, inner sep=6pt},
      line/.style ={draw, thick, -latex', shorten >=0pt},
      line2/.style ={draw, dashed, thick, -latex', shorten >=0pt},
      description/.style={fill=white,inner sep=2pt}]
    \matrix [column sep=50mm,row sep=20mm] {
      \node [block_left] (referred) {Discrete OC Problem \\ $m$ ODEs + $N$ ODEs};
      & \node [block_left] (excluded1) {Continuous OC Problem \\ $m$ ODEs + PDE}; \\
      \node [block_left] (assessment) {PMP \\ $2m$ ODEs + $2N$ ODEs}; 
      & \node [block_left] (excluded2) {Extended PMP \\ $2m$ ODEs + PDE}; \\
    };
    \begin{scope}[every path/.style=line]
      \path (referred) edge[above] node {$N \rightarrow +\infty$}  (excluded1);
      \path (referred) edge[above] node[description] {optimization} (assessment);
      \path (assessment) edge[above] node {$N \rightarrow +\infty$} (excluded2);
    \end{scope}
    \begin{scope}[every path/.style=line2]
      \path (excluded1) edge[above] node[description] {optimization} (excluded2);
    \end{scope}
  \end{tikzpicture}
\end{center}
We shall provide a set of hypotheses for which the dashed line from the upper-right to the bottom-right box is valid. Our strategy shall be the following: we apply the Pontryagin Maximum Principle (see e.g. \cite[Theorem 23.11]{clarke2013functional}) to the finite-dimensional optimal control problems (the solid line from the upper-left to the bottom-left box), and we pass to the mean-field limit the system of equations obtained with this procedure (the solid line from the bottom-left to the bottom-right box). The derived limit equation for the state and the (rescaled) adjoint variables are obtained in the form of Hamiltonian flows in the Wasserstein space of probability measures, in the sense of \cite{ambrosio}. The result will be a first-order condition valid for all \emph{mean-field optimal controls}. The existence of such controls is also proved (see Corollary \ref{c-gamma}).

More formally, we are interested in deriving optimality conditions for the solutions of the following optimal control problem subject to coupled ODE-PDE constraints.

\begin{framed}
\begin{problem}\label{problemPDE}
For $T>0$ fixed, find $u^* \in L^1([0,T];\U)$ minimizing the cost functional
\begin{align} \label{Ffun}
F(u) = \int^T_0 \left[ L(y(t),\mu(t)) + \gamma(u(t)) \right] dt,
\end{align}
where $(y,\mu)$ solve
\begin{align} \label{eq:mfdyn}
\begin{cases}
\dot{y}_k(t) = (K \star \mu(t))(y_k(t)) + f_k(y(t))+B_k u(t), \quad k = 1, \ldots, m, \\
\partial_t \mu(t) =- \nabla_x \cdot \left[(K \star \mu(t) + g(y(t)))\mu(t) \right],
\end{cases}
\end{align}
for the given initial datum $(y(0),\mu(0)) = (y^0,\mu^0) \in \R^{dm} \times \Pc(\R^d)$.
\end{problem}
\end{framed}

Here, $\gamma$ is a strictly convex cost functional, the finite dimensional set of controls $\U$ is convex and compact, $B_k$ are constant matrices, and $\Pc(\R^d)$ is the set of probability measures on $\R^d$ with compact support.



We shall prove the following main result.

\begin{theorem}\label{t-main} Fix an initial datum $(y^0,\mu^0) \in \R^{dm} \times \Pc(\R^d)$ and assume that Hypotheses (H)  in Section \ref{s-hp} hold. Then there exists a mean-field optimal control for Problem \ref{problemPDE}. Furthermore, if $u^*$ is a mean-field optimal control for Problem \ref{problemPDE} and $(y^*,\mu^*)$ is the corresponding trajectory, then $(u^*,y^*,\mu^*)$ satisfies the following {\bf extended Pontryagin Maximum Principle}:

\begin{framed} \noindent There exists $(q^*(\cdot),\nu^*(\cdot)) \in \Lip([0,T]; \R^{dm} \times \mathcal{P}_1(\R^{2d}))$ such that
\bi
\i there exists $R_T > 0$, depending only on $y^0, \supp(\mu^0), d, K, g, f_k, B_k, \mathcal{U},$ and $T$, such that $\supp(\nu^*(\cdot)) \subseteq B(0,R_T)$ and it satisfies $\nu^*(t)(E\times \R^d)=\mu^*(t)(E)$ for all $t\in[0,T]$ and for every Borel set $E \subseteq \R^d$;
\i it holds
\begin{align}\label{e-PMP}
\begin{cases}
\dot y^*_k&=\nabla_{q_k} \H_c(y^*,q^*,\nu^*,u^*),\\
\dot q^*_k&=-\nabla_{y_k}  \H_c(y^*,q^*,\nu^*,u^*),\\
\partial_t\nu^*\!\!\!\!\!\!&=-\nabla_{(x,r)}\cdot \Pt{(J\nabla_{\nu}\H_c(y^*,q^*,\nu^*,u^*))\nu^*}, \\
u^* &= \arg \max_{u \in \mathcal{U}} \H_c(y^*,q^*,\nu^*,u)
\end{cases}
\end{align}
where $J \in \R^{2d \times 2d}$ is the symplectic matrix $$J=\Pt{\ba{cc} 0 & \Id\\ -\Id&0\ea},$$ the Hamiltonian $\H_c: \R^{2dm} \times \Pc(\R^{2d}) \times \R^D \to \R$ is defined as
\bqn
\H_c(y,q,\nu,u) = \begin{cases}
\H(y,q,\nu,u) & \text{ if } \supp(\nu) \subseteq \overline{B(0,R_T)}, \\
+\infty & \text{ elsewhere;}\end{cases}
\eqnn
and $\H: \R^{2dm} \times \Pc(\R^{2d}) \times \R^D \to \R$ is defined as
\begin{align}
\begin{split}\label{e-H}
\H(y,q,\nu,u) = & \frac12\int_{\R^{4d}}(r-r') \cdot K(x-x')\,d\nu(x,r)\,d\nu(x',r') +\int_{\R^{2d}} r \cdot g(y)(x) d\nu(x,r)\\
&+\sum_{k=1}^m \int_{\R^{2d}} q_k \cdot K(y_k-x)\,d\nu(x,r)+\sum_{k=1}^m q_k \cdot (f_k(y)+B_k u)-L(y,\pi_{1\#}\nu)- \gamma(u).
\end{split}
\end{align}
\i the following conditions for system \eqref{e-PMP} hold at time $0$: $y^*(0)=y^0$ and $\nu^*(0)(E\times \R^d)=\mu^0(E)$ for every Borel set $E \subseteq \R^d$, 
\i the following conditions for system \eqref{e-PMP} hold at time $T$: $q^*(T)=0$ and $\nu^*(T)(\R^d\times E)=\delta_0(E)$ for every Borel set $E \subseteq \R^d$, where $\delta_0$ is the Dirac measure centered in $0$.
\ei
\end{framed}

\end{theorem}

As already mentioned, the formulation given above shows that the dynamics of $(y^*,q^*,\nu^*)$ is essentially an Hamiltonian flow in the Wasserstein space of probability measures with respect to state and adjoint variables with Hamiltonian $\H$, in the sense of \cite{ambrosio}. The definition of $\H_c$ is introduced to simplify some technical details and does not alter the result. This fact is remarkably consistent with the dynamics \eqref{eq:mfdyn}, since both are flows in a Wasserstein space. We believe that this formulation of the optimality conditions making use of the formalism of subdifferential calculus in Wasserstein spaces of probability measures constitutes one of the novelties of the work.

\begin{remark} \label{r-uniquemax}
For every $(y,q,\nu)$ with $\supp(\nu) \subseteq \overline{B(0,R_T)}$, \eqref{e-H} immediately implies that
\begin{align*}
\overline{u} \in \arg \max_{u \in \mathcal{U}} \H_c(y, q, \nu, u) \iff \overline{u} \in \arg \max_{u \in \mathcal{U}}\Pt{ \sum_{k=1}^m q_k \cdot B_k u - \gamma(u)}.
\end{align*}
Then, the strict convexity of $\gamma$ and the convexity and the compactness of $\mathcal{U}$ imply that $\overline{u}$ is uniquely determined by $(y,q,\nu)$. This is the reason why we write the equality symbol in $u^* = \arg \max_{u \in \mathcal{U}} \H_c(y^*,q^*,\nu^*,u)$ in place of an inclusion.
\end{remark}

We point out the difference between the usual gradient in $\R^{2d}$ with respect to the state variables $x$ and the adjoint variables $r$, denoted by $\nabla_{(x,r)}$, and the Wasserstein gradient $\nabla_\nu$ of $\H_c$, which, as shown in Section \ref{s-wass}, whenever $\nu$ has supported contained in $B(0,R_T)$ can be computed explicitly as follows:
\begin{itemize}
\item For $l=1,\ldots,d$, it holds
\begin{align}
\begin{aligned}\label{e-gradx}
\nabla_\nu\H_c(y,q,\nu,u)(x,r)\cdot e_l&=\int_{\R^{2d}} (r-r')\cdot(\Jac K(x-x') e_l) \, d\nu(x',r')+ r \cdot (\Jac_x g(y)(x) e_l)\\
&-\sum_{k=1}^m q_k \cdot (\Jac K(y_k-x) e_l)-\nabla_{\xi}\ell(y,x,\mbox{$\int \omega \mu$})\cdot e_l \\
& -\left(\nabla_{\varsigma}\ell(y,x,\mbox{$\int \omega \mu$}) \Jac\omega(x)\right)\cdot e_l.
\end{aligned}
\end{align}
These are the components of $\nabla_\nu\H_c(y,q,\nu,u)(x,r)$ in the $x_l$ coordinates.
\item For $l=d+1,\ldots,2d$ it holds
\bqn
\nabla_\nu\H_c(y,q,\nu,u)(x,r)\cdot e_l=\int_{\R^{2d}} K(x-x')\cdot e_{l-d} \, d\nu(x',r')+g(y)(x)\cdot e_{l-d}.
\eqnl{e-gradr}
These are the components of $\nabla_\nu\H_c(y,q,\nu,u)(x,r)$ in the $r_{l-d}$ coordinates.
\end{itemize}
In \r{e-gradx} and \r{e-gradr}, the functions $\ell\in \cal{C}^2(\R^{dm} \times \R^d \times \R^d;\R)$ and $\omega \in \cal{C}^2(\R^d;\R^d)$ are related to the functional $L$ in \eqref{Ffun} via
$$L(y,\mu) = \int_{\R^{d}} \ell \left(y,x, \mathsmaller \int \om \mu \right) d\mu(x),$$
where $\int \! \om \mu := \om \mu(\R^d)$, 
while $\nabla_{\xi}\ell$ and $\nabla_{\varsigma}\ell$ denote the partial derivatives of the function $\ell(\eta, \xi, \varsigma)$, and $\Jac\omega(x)$ is the Jacobian of the function $\omega$ evaluated at $x$. Notice that $\nabla_{\nu} \H(y,q,\nu,u)$ actually does not depend on $u$, as a consequence of the fact that the control does not act directly on the PDE component of \eqref{eq:mfdyn}.

The main tool we use to prove Theorem \ref{t-main} is the Pontryagin Maximum Principle (henceforth, simply addressed as PMP) for optimal control problems with ODE constraint. We shall apply it to the following finite-dimensional problems, whose constraints converge to the coupled ODE-PDE system of Problem \ref{problemPDE}, as we will show in Section \ref{s-coupled}. For this reason, we call Theorem \ref{t-main} the \emph{extended PMP}.

\begin{framed}
\begin{problem}\label{problemODE}
For $T>0$ fixed, find $u^* \in L^1([0,T];\U)$ minimizing the cost functional
\begin{align} \label{FNfun}
F_N(u) = \int^T_0 \left[ L(y(t),\mu_N(t)) + \gamma(u(t)) \right] dt,
\end{align}
where $(y,\mu_N)$ solve
\begin{align} \label{eq:discdyn}
\begin{cases}
\dot{y}_k = \frac{1}{N}\sum^N_{j = 1}K(y_k - x_j) + f_k(y)+B_k u, \quad k = 1, \ldots, m \\
\dot{x}_i = \frac{1}{N}\sum^N_{j = 1}K(x_i - x_j) + g(y)(x_i), \quad \quad \quad \; \; i = 1, \ldots, N,
\end{cases}
\end{align}
for the given initial datum $(y(0),x(0)) = (y^0,x^0) \in \R^{dm}\times\R^{dN}$, where
\begin{align*}
\mu_N(t)(x) = \frac{1}{N}\sum^N_{i = 1} \delta(x -x_i(t)),
\end{align*}
is the empirical measure centered on the trajectory $x(\cdot) = (x_1(\cdot), \ldots, x_N(\cdot))$.
\end{problem}
\end{framed}

The extended PMP will be derived after reformulating the finite-dimensional PMP applied to Problem \ref{problemODE} in terms of the empirical measure in the product space of state variables $x_i$ and adjoint variables $p_i$, defined as
$$\nu_N(x,r) = \frac{1}{N}\sum^N_{i = 1} \delta(x - x_i, r - Np_i).$$
Notice that rescaling the adjoint variables $p_i$ by the number $N$ of agents is needed in order to observe a nontrivial dynamics in the limit (see also Remark \ref{rem3}); indeed, within this scaling, the right-hand side of the finite-dimensional PMP is brought back to the form considered, for instance, in \cite{CGP}, with a different Hamiltonian.

The following diagram recollects the strategy of the proof, making use of the notation already introduced and reporting in which part of the paper each result is proved:

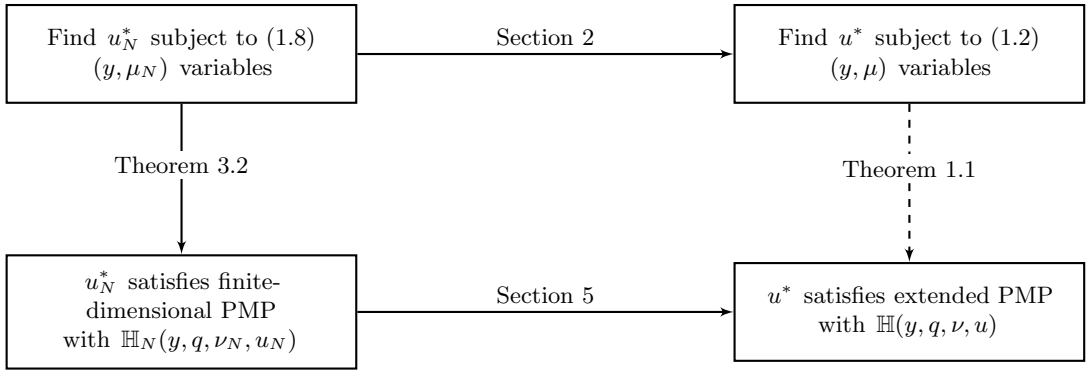
\begin{center}
  \footnotesize
  \begin{tikzpicture}[auto,
    block_center/.style ={rectangle, draw=black, thick, fill=white,
      text width=8em, text centered,
      minimum height=4em},
    block_left/.style ={rectangle, draw=black, thick, fill=white,
      text width=13em,  text centered, minimum height=4em, inner sep=6pt},
    block_noborder/.style ={rectangle, draw=none, thick, fill=none,
      text width=18em, text centered, minimum height=1em},
    block_assign/.style ={rectangle, draw=black, thick, fill=white,
      text width=18em, text ragged, minimum height=3em, inner sep=6pt},
    block_lost/.style ={rectangle, draw=black, thick, fill=white,
      text width=16em, text ragged, minimum height=3em, inner sep=6pt},
      line/.style ={draw, thick, -latex', shorten >=0pt},
      line2/.style ={draw, dashed, thick, -latex', shorten >=0pt},
      description/.style={fill=white,inner sep=2pt}]
    \matrix [column sep=50mm,row sep=20mm] {
      \node [block_left] (referred) {Find $u^*_N$ subject to \eqref{eq:discdyn} \\ $(y,\mu_N)$ variables};
      & \node [block_left] (excluded1) {Find $u^*$ subject to \eqref{eq:mfdyn} \\ $(y,\mu)$ variables}; \\
      \node [block_left] (assessment) {$u^*_N$ satisfies finite-dimensional PMP \\ with $\H_N(y,q,\nu_N,u_N)$}; 
      & \node [block_left] (excluded2) {$u^*$ satisfies extended PMP \\ with $\H(y,q,\nu,u)$}; \\
    };
    \begin{scope}[every path/.style=line]
      \path (referred) edge[above] node {Section \ref{s-coupled}}  (excluded1);
      \path (referred) edge[above] node[description] {Theorem \ref{PMPODE}} (assessment);
      \path (assessment) edge[above] node {Section \ref{s-proof}} (excluded2);
    \end{scope}
    \begin{scope}[every path/.style=line2]
      \path (excluded1) edge[above] node[description] {Theorem \ref{t-main}}  (excluded2);
    \end{scope}
  \end{tikzpicture}
\end{center}

The structure of the paper is the following. In Section \ref{s-hp} we recall notations and the main Hypotheses (H). In Section \ref{s-coupled}, we study the controlled dynamics subject to a coupled ODE-PDE constraint of the form \r{eq:mfdyn}, establishing existence and uniqueness results for solutions. In Section \ref{s-ODE} we study the finite-dimensional Problem \ref{problemODE}, and apply the PMP to it. In Section \ref{s-wass}, we recall basic facts about subdifferential calculus in Wasserstein spaces, and we explicitly compute $\nabla_\nu\H_c$. In Section \ref{s-proof}, we prove the extended PMP, i.e., Theorem \ref{t-main}. Finally, Section \ref{s-CS} is devoted to the study of an interesting example of Problem \ref{problemPDE}, the Cucker-Smale system.

\subsection{Notation and Hypotheses (H)} \label{s-hp}

We start this section by recalling the notation used throughout the paper.

The constants $d,D$ are two positive integers (the dimension of the space of the agents and of the control, respectively), $T > 0$ (the end time of the optimization procedure), and $\U$ is a {\em convex compact} subset of $\R^D$ (set in which controls take values).

Functionals have the following expressions: $K: \R^d \to \R^d$, each $f_k$ satisfies $f_k: \R^{dm} \to \R^d$, and for every $y \in \R^{dm}$ and $\mu \in \mathcal{P}_1(\R^d)$, $g(y): \R^d \to \R^d$ and $L(y,\mu): \R^d \to \R$. The matrices $B_k$ are constant $d\times D$ matrices.

The space $\mathcal{P}(\R^n)$ is the set of probability measures which take values on $\R^n$, while the space\footnote{We follow the notation of \cite{AGS}.} $\mathcal{P}_p(\R^n)$ is the subset of $\mathcal{P}(\R^n)$ whose elements have finite $p$-th moment, i.e.,
$$\int_{\R^n} \|x\|^p d\mu(x) < +\infty.$$
We denote by $\mathcal{P}_c(\R^n)$ the subset of $\mathcal{P}_1(\R^n)$ which consists of all probability measures with compact support. Notice that, if $(\mu_n)_{n \in \N}$ is a sequence in $\mathcal{P}_c(\R^n)$ and it exists $R>0$ such that $\supp(\mu_n) \subseteq B(0, R)$ for all $n\in\N$, then $(\mu_n)_{n \in \N}$ is compact in $\mathcal{P}_p(\R^n)$ for all $p \geq 1$.

For any $\mu \in \mathcal{P}(\R^n)$ and any Borel function $r: \R^{n_1} \to \R^{n_2}$, we denote by $r_{\#}\mu \in \mathcal{P}(\R^{n_2})$ the {\it push-forward of $\mu$ through $r$}, defined by
\bqn
r_{\#}\mu(B) := \mu(r^{-1}(B)) \quad \text{ for every Borel set } B \text{ of } \R^{n_2}.
\eqnn
In particular, if one considers the projection operators $\pi_1$ and $\pi_2$ defined on the product space $\R^{n_1} \times \R^{n_2}$, for every $\rho \in \mathcal{P}(\R^{n_1} \times \R^{n_2})$ we call {\it first} (resp., {\it second}) {\it marginal} of $\rho$ the probability measure $\pi_{1\#}\rho$ (resp., $\pi_{2\#}\rho$). Given $\mu \in \mathcal{P}(\R^{n_1})$ and $\nu \in \mathcal{P}(\R^{n_2})$, we denote with $\Gamma(\mu, \nu)$ the subset of all probability measures in $\mathcal{P}(\R^{n_1} \times \R^{n_2})$ with first marginal $\mu$ and second marginal $\nu$.

On the set $\mathcal{P}_p(\R^n)$ we shall consider the following distance, called the {\it Wasserstein or Monge-Kantorovich-Rubinstein distance},
\bqn  \label{e_Wp}
\W^p_p(\mu,\nu)=\inf \left \{ \int_{\R^{2n}} \|x-y\|^p d \rho(x,y) : \rho \in \Gamma(\mu,\nu) \right \}.
\eqn
If $p = 1$ we have the following equivalent expression for the Wasserstein distance:
\bqn
\W_1(\mu,\nu)=\sup \left \{ \int_{\R^n} \varphi(x) d (\mu-\nu)(x)  : \varphi \in \operatorname{Lip}(\R^n), \; \operatorname{Lip}(\varphi) \leq 1 \right \}.
\eqnn
We denote by $\Gamma_o(\mu,\nu)$ the set of optimal plans for which the minimum is attained, i.e.,
\bqn
\rho \in \Gamma_o(\mu, \nu) \iff \rho \in \Gamma(\mu, \nu) \text{ and } \int_{\R^{2n}} \| x - y \|^p d \rho(x,y) = \W^p_p(\mu,\nu).
\eqnn
It is well-known that $\Gamma_o(\mu, \nu)$ is non-empty for every $(\mu,\nu) \in \mathcal{P}_p(\R^n)\times\mathcal{P}_p(\R^n)$, hence the infimum in \eqref{e_Wp} is actually a minimum. For more details, see e.g. \cite{villani,AGS}.

For any $\mu \in \PP(\R^d)$ and $K: \R^d \to \R^d$, the notation $K \star \mu$ stands for the convolution of $K$ and $\mu$, i.e.,
\bqn
(K \star \mu)(x) = \int_{\R^d} K(x - x') d\mu(x');
\eqnn
this quantity is well-defined whenever $K$ is continuous and \emph{sublinear}, i.e., there exists $C$ such that $\| K(\xi) \| \leq C (1 + \|\xi\|)$ for all $\xi \in \R^d$. Furthermore we shall deal also with the convolution $(\nabla_{(x',r')}\Pa{r',K(x')})\star \nu$ in $\R^{2d}$, whose explicit expression is
\bqn
\Pt{(\nabla_{(x',r')}\Pa{r',K(x')})\star \nu}(x,r) = \int_{\R^{2d}} \Pt{\nabla_{(x',r')}\Pa{r-r',K(x - x')}} d\nu(x',r').
\eqnn
Notice that, under the hypotheses we are going to make, this convolution is not always well-defined for $\nu \in \PP(\R^{2d})$. It is nonetheless well-defined for measures $\nu\in \mathcal{P}_c(\R^{2d})$, that is to say for all the cases that will appear in the sequel.

We shall denote with $\cal{M}_b(\R^{n_1};\R^{n_2})$ the space of bounded Radon vector measures from $\R^{n_1}$ to $\R^{n_2}$, and with $\|\cdot\|_{\cal{M}_b(\R^{n_1};\R^{n_2})}$ the total variation norm on it. If $\om \in \cal{C}(\R^d;\R^d)$ is sublinear and $\mu \in \mathcal{P}_1(\R^d)$, the Radon measure $\om \mu \in \cal{M}_b(\R^d;\R^d)$ is defined as
\begin{align*}
\om \mu(E) := \int_{E} \om(x) d\mu(x), \quad \text{ for every } E \subset \R^d \text{ bounded.}
\end{align*}
We shall denote with $\int \! \om \mu := \om \mu(\R^d)$.

In what follows, we shall consider the space ${\mathcal X}:=\R^{dm}\times\PP(\R^d)$, together with the following distance
\bqn
\|(y,\mu)-(y',\mu')\|_{\mathcal X}:=\|y-y'\|+\W_1(\mu,\mu'),
\eqnl{e-dX}
where $\|y-y'\|:=\sum_{k=1}^m \|y_k-y'_k\|_{\ell_2(\R^d)}$.

Finally, for every $N \in \N$, the mapping $\Pi_N:\R^{2dN}\to\mathcal{P}_1(\R^{2d})$ is defined as follows 
\bqn
\Pi_N: (x_1,p_1,\ldots,x_N,p_N)\mapsto \frac1N \sum_{i=1}^N\delta(\cdot -x_i, \cdot -Np_i).
\eqnl{e-PiNint}

Henceforth, we assume that the following regularity properties hold.
\begin{framed}
\begin{center}
{\bf Hypotheses (H)}
\end{center}
\begin{description}
\item[(K)] The function $K \in \cal{C}^2(\R^d;\R^d)$ is odd and sublinear, i.e., there exists $C_K > 0$ such that for all $x \in \R^d$ it holds
$$\| K(x) \| < C_K(1 + \|x\|).$$
\item[(L)] The function $L: \R^{dm} \times \PP(\R^d) \rightarrow \R$ is
$$L(y,\mu) = \int_{\R^{d}} \ell \left(y,x, \mathsmaller \int \om \mu \right) d\mu(x),$$ with $\ell\in \cal{C}^2(\R^{dm} \times \R^d \times \R^d;\R)$ and $\omega \in \cal{C}^2(\R^d;\R^d)$.
\item[(G)] The function $g \in \cal{C}^2(\R^{dm};\cal{C}^2(\R^d;\R^d))$ satisfies for all $x \in \R^d$ and all $y \in \R^{dm}$
$$g(y)(x) \cdot x \leq G_1 \|x \|^2 + G_2 \max_{l = 1, \ldots, m} \|y_{l}\|^2 + G_3,$$
where the constants $G_1, G_2$ and $G_3$ are independent on $x$ and $y$.
\item[(F)] For each $k = 1, \ldots, m$, the function $f_k \in \cal{C}^2(\R^{dm};\R^d)$ satisfies for all $y \in \R^{dm}$
$$f_k(y) \cdot y_k \leq F_1 \max_{l = 1, \ldots, m} \|y_{l}\|^2+ F_2,$$
where the constants $F_1$ and $F_2$ are independent on $y$ and $k$.
\item[(U)] The set $\mathcal{U}\subseteq\R^D$ is compact and convex.
\item[($\gamma$)] The function $\gamma: \mathcal{U} \rightarrow \R$ is strictly convex.
\end{description}
\end{framed}

\begin{remark}
We briefly compare Hypotheses (H) with those of \cite{andersson2011maximum,burger2013mean}. In \cite{andersson2011maximum}, which deals with an SDE-constrained optimal control problem, $\mathcal{C}^{1,1}$ functionals with respect to state variables and the control are considered. Therefore our hypotheses are just slightly more restrictive. On the other hand, we do not require differentiability of the running cost. The authors of \cite{burger2013mean} deal, instead, with a mean-field game type optimality conditions to model evacuation scenarios. They derive a first-order condition under the hypotheses of continuous differentiability of the functionals with respect to the state variables together with convexity and positivity assumptions. Furthermore, they deal specifically with an $L^2$ control cost, while we allow ours to be strictly convex.
\end{remark}

We now give the rigorous definition of \emph{mean-field optimal control}.

\begin{definition}\label{d-mfoc}
Let $(y^0,\mu^0) \in \R^{dm}\times\Pc(\R^{d})$ be given. An optimal control $u^*$ for Problem \ref{problemPDE} with initial datum $(y^0,\mu^0)$ is a  \emph{mean-field optimal control} if there exists a sequence $(u_N)_{N\in \N} \subset L^1([0,T];\mathcal{U})$ and a sequence $(\mu^0_N)_{N \in \N} \in \Pc(\R^d)$ such that
\begin{enumerate}[label=$(\roman*)$]
\item\label{uno} for every $N \in \N$, $\mu_N^0(\cdot):=\frac1N\sum_{i=1}^N (\cdot-x_{i,N}^0)$ is a sequence of empirical measures for some $x_{i,N}^0\in\supp(\mu^0)+\overline{B(0,1)}$ such that $\mu_N^0 \rightharpoonup \mu^0$ weakly$^*$ in the sense of measures;
\item for every $N \in \N$, $u^*_N$ is a solution of Problem \ref{problemODE} with initial datum $(y^0,\mu^0_N)$;
\item there exists a subsequence of $(u_N)_{N\in \N}$ converging weakly in $L^1([0,T];\U)$ to $u^*$.
\end{enumerate}
\end{definition}

\begin{remark}
As mentioned before, the above definition is motivated by our interest in optimizers that are close to optimal controls for the original finite-dimensional problems.
Notice also, that since the measures $\mu^0_N$ have all compact support contained in $\supp(\mu^0)+\overline{B(0,1)}$, the build a compact sequence in $\mathcal{P}_p(\R^n)$ for all $p \geq 1$, and therefore, due to weak$^*$ convergence to $\mu^0$,  we also have that $\lim_{N\to\infty}\W_p(\mu_N^0,\mu^0)=0$.
\end{remark}

\section{The coupled ODE-PDE dynamics}\label{s-coupled}

In this section, we first recall results for PDE equations of transport type with nonlocal interaction velocities, like the one appearing in the second equation of  \r{eq:mfdyn}. We then study the coupled ODE-PDE dynamics \r{eq:mfdyn} and we state existence and uniqueness results of solutions, together with continuous dependence on the initial data $(y^0,\mu^0)$ and on the control $u$. The proofs follow closely in the footsteps of similar results in \cite{ambrosio,fornasier2014mean,pedestrian,gw}. We also show that finite-dimensional ODE dynamics \r{eq:discdyn} are embedded in \r{eq:mfdyn}, in the sense that the solution of \r{eq:mfdyn} with an initial data that is an empirical measure coincides with the empirical measure with support on the solution of \r{eq:discdyn}.

\subsection{Transport PDE equations with nonlocal interaction}
In this section, we study  equations for the dynamics of measures, recalling results of existence and uniqueness. We first define the meaning of solution for the equation
\bqn
\partial_t\mu(t)=-\nabla_x\cdot (v(t,x,\mu(t)) \mu(t)),
\eqnl{e-pde2}
where $v:[0,T]\times\R^n\times\PP(\R^n)\to\R^n$ is a given vector field and $n \in \N$ is the dimension of the underlying Euclidean space.

\bdeff \label{d-pde12} We say that a map $\mu:[0,T]\to\PP(\R^n)$ is a solution of \r{e-pde2} if the following holds:
\begin{enumerate}[label=$(\roman*)$]
\item $\mu$ has uniformly compact support, i.e., there exists $R > 0$ such that $\supp(\mu(\cdot))\in B(0,R)$;
\item $\mu$ is continuous with respect to the Wasserstein distance $\W_1$;
\item $\mu$ satisfies \r{e-pde2} in the weak sense, i.e. (see \cite[Equation (8.1.4)]{AGS}), 
\bqn
\frac{d}{dt}\int_{\R^{n}}\phi(x)\,d\mu(t)(x)= \int_{\R^{n}}\nabla \phi(x)\cdot v (t,x,\mu(t))\,d\mu(t)(x),
\eqnn
for every $\phi \in \mathcal{C}^{\infty}_c(\R^n;\R)$.
\end{enumerate}
\edeff

Now, we can formally define the concept of solution of the controlled ODE-PDE system \eqref{eq:mfdyn}, which applies, \emph{mutatis mutandis}, to system \eqref{e-PMP} as well.
\bdeff Let $u\in L^1([0,T];\U)$ and $(y^0,\mu^0)\in \mathcal X$, with $\mu^0$ of bounded support, be given. We say that a map $(y,\mu):[0,T]\to \mathcal X$ is a solution of the  system \r{eq:mfdyn} 
with control $u$ if
\begin{enumerate}[label=$(\roman*)$]
\item $(y(0),\mu(0)) = (y^0,\mu^0)$;
\item the solution is continuous {in time} with respect to the metric \r{e-dX} in $\mathcal X$;
\item the $y$ coordinates define a Carath{\'e}odory solution of the following controlled ODE problem 
\begin{equation*}
\dot{y}_k(t) = (K \star \mu(t))(y_k(t)) + f_k(y(t))+B_k u(t), \quad k = 1, \ldots, m,
\end{equation*}
for all $t \in [0,T]$;
\item $\mu$ is a solution of \eqref{e-pde2}, where $v:[0,T]\times\R^d\times\PP(\R^d)\to\R^d$ is the time-varying vector field defined as follows
\begin{equation*}
v(t,x,\mu(t))(x):=(K \star \mu(t) + g(y(t)))(x).
\end{equation*}
\end{enumerate}
\edeff

We now derive the existence of solutions of \eqref{eq:mfdyn} as limits for $N \to \infty$ of the system of ODE \r{eq:discdyn}. We first prove that solutions of \r{eq:discdyn} coincide with specific solutions of \r{eq:mfdyn}. We then prove the limit result with the help of Lemmata \ref{p-estkernel} and \ref{p-lipkernel}.

\bp
Let $N$ be fixed, and the control $u\in L^1([0,T];\U)$ be given. Let $(y,x_N):[0,T]\to \mathcal X$ be the corresponding solution of \r{eq:discdyn}, with $x_N(t)=(x_{1,N}(t),\ldots,x_{N,N}(t))$. Then, the couple $(y,\mu_N):[0,T]\to \R^{dm+dN}$, with $\mu_N(t)$ being the empirical measure
$$\mu_N(t)(x):=\frac{1}N \sum_{i=1}^N (x-x_{i,N}(t)),$$
is a solution of \r{eq:mfdyn} with control $u$.
\ep
\bproof It can be easily proved by rewriting \r{eq:mfdyn} with $\mu_N$ and arguing exactly as in \cite[Lemma 4.3]{MFOC}.
\eproof

\begin{lemma}\label{p-estkernel}
Let $K:\R^d\to\R^d$ satisfy (K) and $\mu \in \PP(\R^d)$. Then for all $y \in \R^d$ it holds
$$\|(K \star \mu)(y)\| \leq C_K\left( 1 + \| y \| + \int_{\R^d} \| x \| d\mu(x) \right).$$
\end{lemma}
\bproof
See, for instance, \cite[Lemma 6.4]{MFOC}.
\eproof

\begin{lemma}\label{p-lipkernel}
Let $K:\R^d\to\R^d$ satisfy (K) and let $\mu^1:[0,T] \rightarrow \Pc(\R^d)$ and $\mu^2: [0,T] \to \PP(\R^d)$ be two continuous maps with respect to $\W_1$ satisfying
$$\supp(\mu^1(t)) \cup \supp(\mu^2(t)) \subseteq B(0,R),$$
for every $t \in [0,T]$, for some $R > 0$. Then for every $\rho > 0$ there exists constant $L_{\rho,R}$ such that
$$\|K \star \mu^1(t) - K \star \mu^2(t)\|_{L^{\infty}(B(0,\rho))} \leq L_{\rho,R} \W_1(\mu^1(t),\mu^2(t))$$
for every $t \in [0,T]$.
\end{lemma}
\bproof
A proof of this result may be found, for instance, in \cite[Lemma 6.7]{MFOC}.
\eproof

\bp \label{p-esistenza}
Let $y^0\in \R^{dm}$, $\mu^0\in \Pc(\R^{d})$, and $\mu_N^0(\cdot)$ be as in Definition \ref{d-mfoc}--\ref{uno}. Let $(u_N)_{N \in \N} \subseteq L^1([0,T];\U)$ be a sequence of controls such that $u_N\weak u$, for some $u\in L^1([0,T];\U)$. 

Then, the sequence of solutions $(y_N,\mu_N) \in \Lip([0,T];{\mathcal X})$ of \r{eq:discdyn} with initial data $(y^0,\mu_N^0)$ and control $u_N$ converges to a solution $(y,\mu) \in \Lip([0,T];{\mathcal X})$ of \r{eq:mfdyn} with initial data $(y^0,\mu^0)$ and control $u$. Moreover, there exists $\rho_T > 0$, depending only on $y^0, \supp(\mu^0), K, g, f_k, B_k, \mathcal{U},$ and $T$, such that for every $N \in \N$, for every $k = 1, \ldots, m$ and for every $t \in [0,T]$ it holds 
$$\|y_{k,N}(t)\|,\|y_k(t)\| \leq \rho_T \quad \mbox{ and } \quad  \supp(\mu_N(t)),\supp(\mu(t)) \subseteq B(0,\rho_T).$$
\ep
\bproof
We start by fixing $N > 0$ and estimating the growth of $\|y_{k,N}(t)\|^2 + \|x_{i,N}(t)\|^2$ for $k = 1, \ldots, m$ and $i = 1, \ldots N$. Let $\Sigma= \{(l,j) : l = 1, \ldots, m \text{ and } j = 1, \ldots N\}$. From Hypotheses (H), Lemma \ref{p-estkernel} and the compactness of $\U$, it holds
\begin{align*}
\frac{1}{2}\frac{d}{dt}\left(\|y_{k,N}\|^2 + \|x_{i,N}\|^2 \right)& =\dot{y}_{k,N}\cdot y_{k,N} + \dot{x}_{i,N} \cdot x_{i,N} \\
& =  \left((K \star \mu_N)(y_{k,N}) + f_k(y)+B_k u\right) \cdot y_{k,N} + \left((K \star \mu_N)(x_i) + g(y)(x_{i,N})\right) \cdot x_{i,N}  \\
& \leq \left\| (K \star \mu_N)(y_{k,N}) \right\| \| y_{k,N} \| + f_k(y_N)\cdot y_{k,N} + \| B_k u \| \| y_{k,N} \| \\
& + \|(K \star \mu_N)(x_{i,N})\| \|x_{i,N}\| +  g(y_N)(x_{i,N})\cdot x_{i,N}  \\
& \leq  C_K\left(1 + \|y_{k,N}\| + \frac{1}{N} \sum^N_{j = 1} \|x_{j,N}\| \right) \|y_{k,N}\| + F_1\max_{l = 1, \ldots m}\|y_{l,N}\|^2 + F_2 + M_1\|y_{k,N}\| \\
&  + C_K\left(1 + \|x_{i,N}\| + \frac{1}{N} \sum^N_{j = 1} \|x_{j,N}\| \right) \|x_{i,N}\| + G_1\|x_{i,N}\|^2 + G_2\max_{l = 1, \ldots m}\|y_{l,N}\|^2 + G_3  \\
& \leq  C_1 \max_{(\ell,j) \in \Sigma}\left\{\|y_{\ell,N}\|^2 + \|x_{j,N}\|^2\right\} + C_2,
\end{align*}
with $C_1 = 4C_K + F_1 + G_2 + M_1$ and $C_2 = C_K + F_2 + G_3 + M_1$. If we denote with $b_{(k,i)}(t) = \|y_{k,N}(t)\|^2 + \|x_{i,N}(t)\|^2$ and with $a(t) = \max_{(l,j) \in \Sigma} \{b_{(l,j)}(t)\}$, then the Lipschitz continuity of $a$ implies that $a$ is a.e. differentiable, while by Stampacchia's Lemma (see for instance \cite[Chapter 2, Lemma A.4]{Kin-Sta}) for a.e. $t \in [0,T]$ there exists a $(l,j) \in \Sigma$ such that
$$\dot{a}(t) = \frac{d}{dt}\left(\|y_{l,N}(t)\|^2 + \|x_{j,N}(t)\|^2 \right) \leq 2C_1 a(t) + 2C_2.$$
Hence, Gronwall's Lemma and Definition \ref{d-mfoc}--\ref{uno} imply that
\begin{align} \label{discgronwall}
a(t) \leq (a(0) + 2C_2t) e^{2C_1t} \leq (C_0 + 2C_2t) e^{2C_1t},
\end{align}
for some uniform constant $C_0$ only depending on $y^0$ and $\supp(\mu^0)$. It then follows that the trajectories $(y_N(\cdot), \mu_N(\cdot))$ are bounded uniformly in $N$ in a ball $B(0,\rho_T) \subset \R^d$, for
$$\rho_T := \sqrt{C_0 + 2C_2T} e^{C_1T},$$
that is positive and does not depend on $t$ or on $N$. This in turn implies that the trajectories $(y_N(\cdot), \mu_N(\cdot))$ are uniformly Lipschitz continuous in $N$, as can be easily verified by computing $\|\dot{y}_{k,N}\|$ and $\|\dot{x}_{i,N}\|$ and noticing that all the functions involved are bounded by Hypotheses (H) and the fact that we are inside $B(0,\rho_T)$. Therefore
\begin{align}\label{discgronwalldot}
\|\dot y_{k,N}(t)\| \le \rho'_T\,,\quad \|\dot x_{i,N}(t)\| \le \rho'_T,
\end{align}
where the constant $\rho'_T$ does not depend on $t$ or on $N$.

By an application of the Ascoli-Arzel\`a theorem for functions on $[0,T]$ and values in the complete metric space ${\mathcal X}$, there exists a subsequence, again denoted by $(y_N(\cdot),\mu_N(\cdot))$ converging uniformly to a limit $(y(\cdot),\mu(\cdot))$, whose trajectories are also contained in $B(0,\rho_T)$.  Due to the equi-Lipschitz continuity of $(y_N(\cdot),\mu_N(\cdot))$ and the continuity of the Wasserstein distance, we thus obtain for some $L_T > 0$
\bqn
\| (y(t_2),\mu(t_2)) - (y(t_1),\mu(t_2)) \|_{\mathcal X} {= \lim_{N \rightarrow +\infty}} \| (y_N(t_2),\mu_N(t_2)) - (y_N(t_1),\mu_N(t_1)) \|_{\mathcal X} \leq L_T |t_2 - t_1|,
\eqnn
for all $t_1,t_2 \in [0,T]$. Hence, the limit trajectory $(y^*(\cdot),\mu^*(\cdot))$ belongs as well to $\Lip([0,T];{\mathcal X})$. 

It is now necessary to show that the limit $(y(\cdot),\mu(\cdot))$ is a solution of \eqref{eq:mfdyn}. We first verify that $y$ is a solution of the ODEs part for $\mu = \mu$. To this end, we observe that the limit $(y_N(\cdot),\mu_N(\cdot)) \to (y(\cdot),\mu(\cdot))$ in $\mathcal{X}$ specifies into 
\begin{equation}\label{unif}
\left \{
\begin{array}{ll}
y_N \rightrightarrows y, & \mbox{ in } [0,T],\\
\dot y_N  \rightharpoonup \dot y, & \mbox{ in } L^1([0,T],\mathbb R^{2d}).\\
\end{array}
\right .
\end{equation}
and 
\begin{equation}\label{convwas}
\lim_{N \to +\infty} \W_1 (\mu_N(t), \mu(t)) =0, 
\end{equation}
uniformly with respect to $t \in [0,T]$. As a consequence of \eqref{unif}, \eqref{convwas}, hypothesis (K), and Lemma \ref{p-lipkernel}, for all $k = 1, \ldots, m$ we have in $[0,T]$ for $N \to +\infty$
\begin{equation}\label{potr}
\begin{array}{clcl}
(K \star \mu_N) (y_{k,N}) &\rightrightarrows &  (K \star \mu) (y_{k}), \\
f_k(y_{N}) &\rightrightarrows &  f_k(y).
\end{array}
\end{equation}

To prove that $y(t)$ is actually the Carath{\'e}odory solution of \eqref{eq:discdyn}, we have only to show that for all $k=1,\dots,m$ one has
\bqn
\dot y_{k} = (K \star \mu)  (y_{k}) + f_k(y)+B_k u.
\eqnn
This is clearly equivalent to the following: for every $\eta \in \R^d$ and every $\hat t\in [0,T]$ it holds
\begin{equation}\label{claim2}
\eta \cdot \int_0^{\hat t}\dot y_{k}(t)\,dt = \eta \cdot \int_0^{\hat t}[( K\star \mu(t))(y_{k}(t)) +  f_k(y(t))+B_k u(t))]\,dt,
\end{equation} 
which follows from \eqref{potr} and from the weak $L^1$-convergence of $\dot y_{k,N}$ to $\dot y_{k}$ and of $u_N$ to $u$ for $N\to +\infty$. 

We are now left with verifying that $\mu$ is a solution of \eqref{eq:mfdyn} for $y = y$.  For all $\hat t \in [0,T]$ and for all $\phi \in C_c^1(\mathbb R^{d};\R)$ we infer that
$$
\langle \phi, \mu_N(\hat t) - \mu_N(0) \rangle = \int_0^{\hat t} \left [ \int_{\mathbb R^{d}} \nabla \phi(x) \cdot [(K \star \mu_N)(x) + g(y_N)(x)] d \mu_N(t)(x) \right ] dt,
$$
which is verified by considering the differentiation
\begin{eqnarray*} \frac{d}{dt} \langle \phi, \mu_N(t)\rangle &=&  \frac{1}{N} \frac{d}{dt} \sum_{i=1}^N \phi(x_i(t))= \frac{1}{N} \left [ \sum_{i=1}^N \nabla \phi(x_i(t)) \cdot \dot x_i(t) \right],
\end{eqnarray*}
and directly applying the substitution $\dot x_i=(K \star \mu_N)(x_i) + g(y_N)(x_i)$.  By Lemma \ref{p-lipkernel} and \eqref{convwas}, we also have that for every $\rho >0$ 
$$
\lim_{N \to +\infty}  \|K\star \mu_{N}(t) -K\star \mu(t)\|_{L^\infty(B(0,\rho))} =0 \mbox{ in } [0,T],
$$
and, as $\phi \in C_c^1(\mathbb R^{2d})$ has compact support, it follows that
$$
\lim_{N\to +\infty}  \|\nabla \phi \cdot (K\star \mu_{N}(t) -K\star \mu(t))\|_\infty =0 \mbox{ in } [0,T].
$$
Similarly, we have
$$\lim_{N\to +\infty}  \|\nabla \phi \cdot (g(y_{N}(t)) - g(y(t)))\|_{\infty} =0 \mbox{ in } [0,T],$$
by the compact support of $\phi$, the $\mathcal{C}^1$-continuity of $g$ and the uniform convergence of $y_{N}$ to $y$. Denote with $\mathcal L^1\llcorner_{[0,\hat t]}$ the Lebesgue measure on the time interval $[0,\hat t]$. Since the product measures $\mathcal L^1\llcorner_{[0,\hat t]} \times \frac{1}{\hat t} \mu_{N}(t)$ converge in $\mathcal P_1([0,\hat t] \times \mathbb R^{2d})$ to $\mathcal L^1\llcorner_{[0,\hat t]} \times \frac{1}{\hat t} \mu(t)$, we finally get
\begin{align*}
\lim_{N \to +\infty} \int_0^{\hat t} \int_{\mathbb R^{d}} \nabla \phi(x) \cdot [K\star \mu_{N}(t)+&g(y_{N}(t)](x) d \mu_{N}(t)(x) dt \\
&=  \int_0^{\hat t} \int_{\mathbb R^{d}} \nabla \phi(x) \cdot [K\star \mu(t)+g(y(t))](x) d \mu(t)(x) dt, 
\end{align*}
that, together with \r{claim2}, proves that $(y,\mu)$ is a solution of \r{eq:discdyn} with initial data $(y^0,\mu^0)$ and control $u$.
\end{proof}

\begin{corollary}
Let $y^0\in \R^{dm}$, $\mu^0\in \Pc(\R^{d})$, and $u\in L^1([0,T];\U)$. Then, there exists a solution of \eqref{eq:mfdyn} with control $u$ and initial datum $(y^0, \mu^0)$.
\end{corollary}
\begin{proof}
Follows from Proposition \ref{p-esistenza} by taking any sequence of empirical measures $\mu^0_N$ as in Definition \ref{d-mfoc}--\ref{uno}, and the constant sequence $u_N \equiv u$ for all $N \in \N$.
\end{proof}

The following intermediate results shall be helpful in proving the continuous dependance on the initial data.

\bp\label{p-kernel} Let $K:\R^d\to\R^d$ and $g: \R^{dm} \rightarrow \mathcal{C}^2(\R^d; \R^d)$ satisfy hypotheses (K) and (G). Then, for every $R > 0$, there exists $L'_R > 0$ satisfying $L'_R \leq C'(1 + R)$ for some $C' > 0$, and
\bqn
\|(K\star\mu^1)(x^1) -(K\star\mu^2)(x^2)\|\leq L'_R (\W_1(\mu^1,\mu^2)+\|x^1-x^2\|),
\eqnl{e-lipK2}
for all $x^1,x^2 \in B(0,R) \subset \R^d$ and $\mu^1, \mu^2 \in \PP(\R^{d})$ with $\supp(\mu^1), \supp(\mu^2) \subseteq B(0,R)$.

Moreover, for every $R > 0$, there exists $L_R > 0$ satisfying $L_R \leq C(1 + R)$ for some $C > 0$, and
\bqn
\|(K\star\mu^1)(x^1) + g(\overline{y})(x^1)-(K\star\mu^2)(x^2) - g(\overline{y})(x^2)\|\leq L_R (\W_1(\mu^1,\mu^2)+\|x^1-x^2\|),
\eqnl{e-lipK}
for all $x^1,x^2 \in B(0,R) \subset \R^d$, $\overline{y} \in B(0,R) \subset \R^{dm}$ and $\mu^1, \mu^2 \in \PP(\R^{d})$ with $\supp(\mu^1), \supp(\mu^2) \subseteq B(0,R)$. 
\ep
\bproof
By hypothesis, we have
\bqn
\|(K\star\mu^1)(x)-(K\star\mu^2)(x)\|= \left\|\int_{\R^n} K(x-x')\,d(\mu^1-\mu^2)(x')\right\|\leq \Lip_{2R}(K) \W_1(\mu^1,\mu^2),
\eqnn and
\bqn
\|(K\star\mu^1)(x^1)-(K\star\mu^1)(x^2)\|\leq\int_{\R^n} \|K(x^1-x)-K(x^2-x)\|\,d\mu^1(x)\leq \Lip_{2R}(K) \|x^1-x^2\|,
\eqnn
where $\Lip_{2R}(K)$ stands for the Lipschitz constant of $K$ on $B(0,2R)$. Since from (K) it follows
$$\Lip_{2R}(K) \leq 2C_K(1 + R),$$
this proves \eqref{e-lipK2} for $L'_R := \Lip_{2R}(K)$ and $C' := C_K$. Moreover, there exists $\xi \in \{tx^1 + (1-t)x^2 : t \in [0,1]\}$ such that
\bqn
\|g(\overline{y})(x^1)-g(\overline{y})(x^2)\|\leq \sup_{\xi \in B(0,R) \subset \R^d, \varsigma \in B(0,R) \subset \R^{dm}}\|\Jac_y g(\varsigma)(\xi)\| \|x^1- x^2\| \leq M \|x^1- x^2\|,
\eqnn
for some $M > 0$, from the regularity of $g$. It then suffices to observe that, for some $C > 0$, it holds
$$\Lip_{2R}(K) + M \leq 2C_K(1 + R) + M \leq C(1 + R).$$
This proves \r{e-lipK} for $L_R = \Lip_{2R}(K) + M$.
\eproof

The estimate in Proposition \ref{p-kernel} shows that the following general result holds for vector fields of the form $v(t,x,\mu(t)) := (K \star \mu(t) + g(y(t)))(x)$, since from Proposition \ref{p-esistenza} follows that $x, y$ and $\mu$ lie in domains with \emph{a priori} known bounds.

\bp \label{p-esistenzapde}
Let $v,w:[0,T]\times\R^d\times\PP(\R^d)\to\R^d$ be vector fields that satisfy the following hypotheses:
\begin{enumerate}
\item $v$ and $w$ are measurable with respect to $t$;
\item for every $R>0$ there exists $L_R$ satisfying $L_R\leq C(1+R)$ such that for all $\mu^1,\mu^2\in\PP(\R^d)$ with support in $B(0,R)$ and all $x^1, x^2 \in \R^d$ it holds
\begin{align}
\begin{split} \label{e-lipmia}
\|v(t,x^1,\mu^1)-v(t,x^2,\mu^2)\| &\leq L_R (\W_1(\mu^1,\mu^2)+\|x^1-x^2\|), \\
\|w(t,x^1,\mu^1)-w(t,x^2,\mu^2)\| &\leq L_R (\W_1(\mu^1,\mu^2)+\|x^1-x^2\|).
\end{split}
\end{align}
\end{enumerate}
Moreover, given $\mu^{0,1},\mu^{0,2}\in\mathcal P_c(\R^d)$, assume that there exist two corresponding solutions $\mu^1,\mu^2$ of \r{e-pde2} with vector fields $v,w$, respectively, and final time $T$. Then there exist constants $C_1$ and $C_2$ such that
\bqn
\W_1(\mu^1(t),\mu^2(t))\leq e^{C_1 t} \W_1(\mu^{0,1},\mu^{0,2})+\int^t_0 C_2e^{C_1 s}\sup_{x\in B(0,R)}\|v(s,x,\mu^1(s))-w(s,x,\mu^2(s))\| \, ds,
\eqnl{e-flowpde}
where $C_1$ and $C_2$ depend on the final time $T$, on the radius $R$ and $L_R$ the Lipschitz constant in \r{e-lipmia}.
\ep
\bproof
See proofs in \cite[Lemma 6.5, Lemma 6.6, Theorem 6.8]{fornasier2014mean}.
\eproof

We now prove the continuous dependence on the initial data, that also gives uniqueness of the solution for \r{eq:mfdyn}.

\begin{proposition} \label{p-unique}
Let the Hypotheses (H) hold. Let $u \in L^1([0,T],\U)$ be given, and take two solutions $(y^1,\mu^1)$ and $(y^2,\mu^2)$ of  \eqref{eq:mfdyn} with control $u$ and with initial data $(y^{0,1},\mu^{0,1}), (y^{0,2},\mu^{0,2}) \in \mathcal X$, respectively, where $\mu^{0,1}$ and $\mu^{0,2}$ have both compact support. Then there exists a constant $C_T>0$ such that
\begin{equation*}
\| (y^1(t),\mu^1(t)) - (y^2(t),\mu^2(t)) \|_{\mathcal X} \leq C_T  \| (y^{0,1},\mu^{0,1}) - (y^{0,2},\mu^{0,2}) \|_{\mathcal X}, \quad \mbox{for all } t \in [0,T]
\end{equation*}
\end{proposition}
\begin{proof}
We start by noticing that, by the definition of a solution, we infer the existence of a $\rho_T > 0$ for which $y^1(\cdot),y^2(\cdot) \in B(0,\rho_T) \subset \R^{dm}$ and $\supp(\mu^1(\cdot)),\supp(\mu^2(\cdot)) \subseteq B(0,\rho_T) \subset \R^d$.

We shall show the continuous dependence estimate by chaining the stability of the ODE
\bqn
\dot{y}_k(t) = (K \star \mu(t))(y_k(t)) + f_k(y(t))+B_k u(t), \quad k = 1, \ldots, m,
\eqnl{e-ode}
with the one of the PDE
\bqn
\partial_t \mu(t) = -\nabla_x \cdot \left[(K \star \mu(t) + g(y(t)))\mu(t) \right],
\eqnl{e-pde}
first addressing the dependence of \eqref{e-ode}. By integration we have 
\begin{align}
\begin{split}\label{firstest}
\|y_k^1(t) - y_k^2(t)\| \leq & \|y_k^{0,1} - y_k^{0,2}\| \\
&+ \int_0^t \left(\|(K\star \mu^1(s))(y_k^1(s))- (K\star\mu^2(s))(y_k^2(s))\| +\|f_k(y^1(s))-f_k(y^2(s))\|\right)\,ds.
\end{split}
\end{align}
For the sake of notation, we shall denote with
\begin{align*}
F = & \max_{k = 1, \ldots, m} \Lip_{\rho_T}(f_k),\\
G = & \sup_{\xi \in B(0,\rho_T) \subset \R^d, \varsigma \in B(0,\rho_T) \subset \R^{dm}}\|\Jac_y g(\varsigma)(\xi)\|.
\end{align*}
For the left-hand side of \eqref{firstest}, \eqref{e-lipK2}, the $\mathcal{C}^2$-regularity of $f_k$ for $k = 1, \ldots, m$, and the uniform bound on $y^1(\cdot)$ and $y^2(\cdot)$ yield
\bqn
\|y_k^1(t) - y_k^2(t)\| &\leq& \|y_k^{0,1} - y_k^{0,2}\| + \label{e-odestima}\\
&&+\int_0^t \left(L'_{\rho_T} \W_1(\mu^1(s),\mu^2(s))+ L'_{\rho_T}\|y^1_k(s)-y_k^2(s))\| +F\|y^1(s)-y^2(s)\| \right) \, ds
\eqnn
We now consider \eqref{e-pde}. Define the vector fields 
$$v^1(t,x,\mu):=(K \star \mu + g(y^1(t)))(x),\qquad\qquad
v^2(t,x,\mu):=(K \star \mu + g(y^2(t)))(x),$$
and let $\Phi:\R^d \to \R$ be a $\mathcal{C}^{\infty}$ cutoff function on $B(0,\rho_T)$ with $\|\nabla \Phi \| \leq 1$ and compact support in $\R^d$. Observe that, since $\|y^1(\cdot)\|,\|y^2(\cdot)\| \leq \rho_T$ and $\supp(\mu^1(\cdot)),\supp(\mu^2(\cdot)) \subseteq B(0,\rho_T)$, then $\mu^1$ and $\mu^2$ also solve \eqref{e-pde2} with $\Phi v^1$ and $\Phi v^2$ in place of $v^1$ and $v^2$, respectively. It then follows easily from Proposition \ref{p-kernel} that Proposition \ref{p-esistenzapde} holds for $v = \Phi v^1$ and $w = \Phi v^2$. Hence, from \eqref{e-flowpde} and taking into account that $v^1 = \Phi v^1$ and $v^2 = \Phi v^2$ in $B(0, \rho_T)$, we have
\bqn
\W_1(\mu^1(t), \mu^2(t)) &\leq& e^{C_1t}\W_1(\mu^{0,1},\mu^{0,2})+\int^t_0 C_2e^{C_1s} \sup_{x\in B(0,\rho_T)} \|v^1(s,x,\mu^1(s))-v^2(s,x,\mu^2(s))\| \, ds,
\eqnn
for some constants $C_1$ and $C_2$. By \r{e-lipK} and the regularity of $g$, for every $s \in [0, T]$ we have
\bqn
\|v^1(s,x,\mu^1(s))-v^2(s,x,\mu^2(s))\|\leq L_{\rho_T} \W_1(\mu^1(s),\mu^2(s))+G\|y^1(s)-y^2(s)\|.
\eqnn
This gives
\bqn
\W_1(\mu^1(t), \mu^2(t)) &\leq& e^{C_1t}\W_1(\mu^{0,1},\mu^{0,2}) \nonumber \\
& & +\int^t_0 C_2e^{C_1s} \Pt{L_{\rho_T} \W_1(\mu^1(s),\mu^2(s))+G\|y^1(s)-y^2(s)\|} \, ds.
\eqnl{e-pdestima}
We now consider the function 
\bqn
\eps(t):= \|(y^1(t),\mu^1(t))-(y^2(t),\mu^2(t))\|_{\mathcal X}
\eqnn
and, combining \r{e-odestima} for each $k = 1,\ldots,m$ and \r{e-pdestima}, we obtain
\begin{align*}
\eps(t) \leq &  \|y^{0,1} - y^{0,2}\| +\int_0^t \left(L'_{\rho_T} \W_1(\mu^1(s),\mu^2(s))+ L'_{\rho_T}\|y^1(s)-y^2(s))\| +mF\|y^1(s)-y^2(s)\| \right) \, ds \\
& + e^{C_1t}\W_1(\mu^{0,1},\mu^{0,2})+\int^t_0 C_2e^{C_1s} \Pt{L_{\rho_T} \W_1(\mu^1(s),\mu^2(s))+G\|y^1(s)-y^2(s)\|} \, ds \\
\leq & \eps(0) e^{C_1t} + \int^t_0 (L'_{\rho_T} + mF + (L_{\rho_T} + G)C_2 e^{C_1s}) \eps(s) \, ds.
\end{align*}
Gronwall's lemma then implies
\bqn
\eps(t)\leq \eps(0) e^{C_1t}\left((L'_{\rho_T} + mF)t + \frac{(L_{\rho_T} + G)C_2}{C_1} (e^{C_1t} - 1)\right).
\eqnn
Since $t \in [0,T]$, the result is proved.
\end{proof}

\begin{remark}
Going back to the application of the Ascoli-Arzel\'a Theorem in Proposition \ref{p-esistenza}, consider another converging subsequence of $(y_N,\mu_N)$. We can prove that its limit is another solution of \r{eq:discdyn}. Since the solution is unique for Proposition \ref{p-unique}, we have that all converging subsequences of $(y_N,\mu_N)$ have the same limit, hence the sequence $(y_N,\mu_N)$ has itself limit $(y^*,\mu^*)$.
\end{remark}

\begin{remark}
Since equicompactly supported solutions are unique, given the initial datum, by Proposition \ref{p-unique}, combined with Proposition \ref{p-esistenza} we infer that the support of the unique solution can be estimated as a function of the data, namely it is contained in a ball $B(0, \rho_T)$, where the constant is depending only on $y^0, \supp(\mu^0), K, g, f_k, B_k, \mathcal{U},$ and $T$.
\end{remark}

\subsection{Existence and construction of solutions of Problem \ref{problemPDE}}

In this section, we prove that Problem \ref{problemPDE} admits a solution which is a mean-field optimal control. The proof generalizes similar results in \cite{fornasier2014mean}.

We first recall the main definition of $\Gamma$-convergence. We then define the sequence of functionals $(F_N)_{N \in \N}$ related to Problem \ref{problemODE} and $F$ related to Problem \ref{problemPDE} and prove that $(F_N)_{N \in \N}$ $\Gamma$-converge to $F$. 

\begin{definition} [\( \Gamma \)-convergence]
  \label{def:gamma-conv}
  \cite[Definition 4.1, Proposition 8.1]{93-Dal_Maso-intro-g-conv}
  Let \( X \) be a metrizable separable space and \( F_N \colon X \rightarrow (-\infty,\infty] \), \( N \in \mathbb{N} \) be a sequence of functionals. We say that \( (F_N)_{N \in \N} \) \emph{\( \Gamma \)-converges} to \( F \), written as \( F_N \xrightarrow{\Gamma} F \), for a given \( F \colon X \rightarrow (-\infty,\infty] \), if
  \begin{enumerate}
  \item \emph{\( \liminf \)-condition:} For every \( u \in X \) and every sequence \( u_N \rightarrow u \),
    \begin{equation*}
      F(u) \leq \liminf_{N\rightarrow+\infty} F_N(u_N);
    \end{equation*}
  \item \emph{\( \limsup \)-condition:} For every \( u \in X \), there exists a sequence \( u_N \rightarrow u \), called \emph{recovery sequence}, such that
    \begin{equation*}
      F(u) \geq \limsup_{N\rightarrow+\infty} F_N(u_N).
    \end{equation*}
  \end{enumerate}

Furthermore, we call the sequence \( (F_N)_{N \in \N} \) \emph{equi-coercive} if for every \( c \in \mathbb{R} \) there is a compact set \( K \subseteq X \) such that \( \left\{ u : F_N(u) \leq c \right\} \subseteq K \) for all \( N \in \mathbb{N} \). As a direct consequence of equi-coercivity,  {assuming \( u_N^* \in \arg \min F_N \neq \emptyset \)} for all $N \in \mathbb N$, there is a subsequence \( (u_{N_k}^*)_{k \in \N} \) and \( u^\ast \in X \) such that
  \begin{equation*}
    u_{N_k}^* \rightarrow u^\ast \in \arg \min F.
  \end{equation*}
\end{definition}

In view of the definition of $\Gamma$-convergence, let us fix as our domain $X = L^1([0,T];\U)$ which, endowed with the weak $L^1$-topology, is actually a metrizable space.

Fix now an initial datum $(y^0,\mu^0) \in \mathcal X$, with $\mu^0$ compactly supported, and choose a sequence $\mu_N^0$ as in Definition \ref{d-mfoc}--\ref{uno}.

Consider the functional $F(u)$ on $X$ defined in \r{Ffun}, where the pair $(y,\mu)$ defines the unique solution of \eqref{eq:mfdyn} with initial datum $(y^0,\mu^0)$ and control $u$. Similarly, consider the functional $F_N(u)$ on $X$ defined in \r{FNfun}, where the pair $(y_N,\mu_N)$ defines the unique solution of \eqref{eq:mfdyn} with initial datum $(y^0,\mu^0_N)$ and control $u$. As recalled in Proposition \ref{p-esistenza}, such solution coincides with the solution of the ODE system \r{eq:discdyn}.

The rest of this section is devoted to the proof of the $\Gamma$-convergence of the sequence of functionals $(F_N)_{N \in \mathbb N}$ on $X$ to the target functional $F$. Let us mention that $\Gamma$-convergence in optimal control problems has been already considered, see for instance \cite{BuDm82}, but, to our knowledge, it has been only recently specified in connection to mean-field limits in \cite{fornasier2014mean,MFOC}.

\begin{theorem}\label{thm:gamma}
Let the functionals \r{Ffun}-\r{FNfun} and dynamics \r{eq:mfdyn} satisfy Hypotheses (H). Consider an initial datum $(y^0,\mu^0) \in \R^{dm}\times \mathcal{P}_1(\R^d)$, and a sequence $(\mu^0_N)_{N \in \N}$, where $\mu_N^0$ is as in Definition \ref{d-mfoc}--\ref{uno}. Then the sequence of functionals $(F_N)_{N \in \mathbb N}$ on $X=L^1([0,T];\U)$ defined in \eqref{FNfun} $\Gamma$-converges to the functional $F$ defined in \eqref{Ffun}.
\end{theorem}
\begin{proof}
Let us start by showing the $\Gamma-\liminf$ condition. Let us fix a weakly convergent sequence of controls $u_N \rightharpoonup u^*$ in $X$. We associate to each of these controls a sequence of solutions $(y_N,\mu_N)$ of \eqref{eq:mfdyn} uniformly convergent to a solution $(y^*,\mu^*)$ with control $u^*$ and initial datum  $(y^0,\mu^0)$. In view of the fact that solutions $(y_N,\mu_N)$ and $(y^*,\mu^*)$ will have uniformly bounded supports with respect to $N$ and $t \in [0,T]$ and by the uniform convergence of trajectories $y_N(t) \rightrightarrows y^*(t)$ as well as the uniform convergence $\W_1(\mu_N(t),\mu^*(t)) \to 0$ for $t \in [0,T]$, it follows from the continuity of $L$ under Hypotheses (H) that
\begin{eqnarray}
&&\lim_{ N \to +\infty} \int_{0}^T   L(y_{N}(t),\mu_{N}(t)) =\int_{0}^T  L (y^{*}(t),\mu^*(t)) dt. \label{lowsem1}
\end{eqnarray}
By the assumed weak convergence of $(u_N)_{N \in \mathbb N}$ to $u^* \in X$ and Ioffe's Theorem (see, for instance, \cite[Theorem 5.8]{AFP00}) we obtain the lower-semicontinuity of $\gamma$
\begin{equation}
\liminf_{N \to +\infty} \int_0^T  \gamma(u_N(t))  dt \geq  \int_{0}^T \gamma(u^*(t)) dt. \label{lowsem2}
\end{equation}
By combining \eqref{lowsem1} and \eqref{lowsem2}, we immediately obtain the $\Gamma-\liminf$ condition
$$
\liminf_{N \to +\infty} F_N(u_N) \geq F(u^*).
$$
We now prove the $\Gamma-\limsup$ condition. We now fix $u^*$ and consider the trivial recovery sequence $u_N \equiv u^*$ for all $N \in \mathbb N$. Similarly as above for the argument of the $\Gamma-\liminf$ condition, we can associate to each of these controls a sequence of solutions $(y_N(t),\mu_N(t))$ of \eqref{eq:mfdyn} uniformly convergent to a solution $(y^*(t),\mu^*(t))$ with control $u^*$ and initial datum  $(y^0,\mu^0)$ and we can similarly conclude the limit \eqref{lowsem1}. Additionally, since $(u_N)_{N \in \mathbb N}$ is a constant sequence, we have
\begin{equation}
\liminf_{N \to +\infty}   \int_0^T  \gamma(u_N(t))  dt = \int_{0}^T \gamma( u^*(t))  dt. \label{lowsem3}
\end{equation}
Hence, combining \eqref{lowsem1} and \eqref{lowsem3} we can easily infer
$$
\limsup_{N \to \infty} F_N(u_N) = \lim_{N \to \infty} F_N(u^*) =  F(u^*).
$$
\end{proof}

\begin{corollary} \label{c-gamma} Let the Hypotheses (H) in Section \ref{s-hp} hold. For every initial datum $(y^0,\mu^0) \in \R^{dm}\times\Pc(\R^d)$, there exists a mean-field optimal control $u^*$ for Problem \ref{problemPDE}.
\end{corollary}
\begin{proof}
Consider empirical measures $\mu_N^0$ as in Definition \ref{d-mfoc}--\ref{uno}. Notice that the optimal controls $u_N^*$ of Problem \ref{problemODE} belong to $X=  L^1([0,T];\U)$, which is a compact set with respect to the weak topology of $L^1$. Hence, the sequence $(F_N)_{N \in \N}$ is equicoercive, and $(u_N^*)_{N \in \mathbb N}$ admits a  subsequence, which we do not relabel, weakly convergent to some $u^* \in X$.

We can associate to each of these controls $u^*_N$ and initial data $(y^0,\mu^0_N)$ a solution $(y_N,\mu_N)$ of \eqref{eq:mfdyn}. The sequence of solutions $(y_N,\mu_N)$ is then uniformly convergent to a solution $(y^*,\mu^*)$ of \eqref{eq:mfdyn} with control $u^*$, by Proposition \ref{p-esistenza}. In order to conclude that $u^*$ is an optimal control for Problem \ref{problemPDE} (and hence, by construction, that $u^*$ is a \emph{mean-field optimal control}) we need to show that it is actually a minimizer of $F$. For that we use the fact that $F$ is the $\Gamma$-limit of the sequence $(F_N)_{N \in \mathbb N}$ as proved in Theorem \ref{thm:gamma}. Let $u \in X$ be an arbitrary control and let $(u_N)_{N \in \mathbb N}$ be a recovery sequence given by the $\Gamma-\limsup$ condition, so that
\begin{eqnarray}
F(u) \geq \limsup_{N \to +\infty} F_N(u_N).\label{firstin}
\end{eqnarray}
By using now the optimality of $(u_N^*)_{N \in \mathbb N}$, we have
\begin{eqnarray}
 \limsup_{N \to +\infty} F_N(u_N) \geq \limsup_{N \to +\infty} F_N(u_N^*) \geq \liminf_{N \to +\infty} F_N(u_N^*).\label{secondin}
\end{eqnarray}
Applying the  $\Gamma-\liminf$ condition yields
\begin{eqnarray}
  \liminf_{N \to +\infty} F_N(u_N^*) \geq F(u^*).\label{thirdin}
\end{eqnarray}
By chaining the inequalities \eqref{firstin}-\eqref{secondin}-\eqref{thirdin} we have 
$$
F(u) \geq F(u^*), \quad \mbox{ for all } u \in X,
$$
i.e., that $u^*$ is an optimal control.
\end{proof}

\begin{remark} Observe that the previous result does not state uniqueness of the optimal control for the infinite dimensional problem. Indeed, in general, we cannot ensure that {\it all} solutions of Problem \ref{problemPDE} are mean-field optimal controls.
\end{remark}

\section{The finite-dimensional problem}
\label{s-ODE}

In this section we study the discrete Problem \ref{problemODE} and state the PMP for it. We first recall the following existence result for the optimal control problem.

\bp[Theorem 23.11, \cite{clarke2013functional}] \label{p-esiODE}
Under Hypotheses (H), Problem \ref{problemODE} admits solutions.
\ep

We now introduce the adjoint variables of $x_i$ and $y_k$, denoted by $p_i$ and $q_k$, respectively, and state the PMP in the following box.

\begin{framed}
\begin{theorem}[Theorem 22.2, \cite{clarke2013functional}]\label{PMPODE}
Let $u^*_N$ be a solution of Problem \ref{problemODE} with initial datum $(y(0),x(0)) = (y^0, x^0)$, and denote with $(y^*(\cdot), x^*(\cdot)): [0,T] \to \R^{dm + dN}$ the corresponding  trajectory. Then there exists a Lipschitz curve $(y^*(\cdot), q^*(\cdot), x^*(\cdot), p^*(\cdot))\in \mathrm{Lip}([0,T],\R^{2dm + 2dN})$ solving the system
\begin{align} \label{eq:PMPdiscrete}
\begin{cases}
\begin{split}
\dot{y}^*_k & = \nabla_{q_k} \mathbb{H}_N(y^*, q^*, x^*, p^*,u^*) \\
\dot{q}^*_k & = - \nabla_{y_k} \mathbb{H}_N(y^*, q^*, x^*, p^*,u^*)
\end{split} \quad k = 1, \ldots, m, \\
\begin{split}
\dot{x}^*_i & = \nabla_{p_i} \mathbb{H}_N(y^*, q^*, x^*, p^*,u^*) \\
\dot{p}^*_i & = - \nabla_{x_i} \mathbb{H}_N(y^*, q^*, x^*, p^*,u^*)
\end{split} \quad i = 1, \ldots, N,\\
\begin{split}
\!\!\!\!\!\!\!\!u^*_N  = \arg \max_{u \in \U} \H_N(y^*, q^*, x^*, p^*, u), \\
\end{split}
\end{cases}
\end{align}
with initial datum $(y(0),x(0)) = (y^0, x^0)$ and terminal datum $(q(T),p(T)) = 0$, where the Hamiltonian $\H_N: \R^{2dm + 2dN} \to \R$ is given by
\begin{align}
\begin{split}\label{e-u}
\H_N(y, q,x, p, u) & = \sum^N_{i = 1} p_i \cdot \left(\frac{1}{N}\sum^N_{j = 1}K(x_i - x_j) + g(y)(x_i)\right) + \\
&\quad + \sum^m_{k = 1} q_k \cdot \left(\frac{1}{N}\sum^N_{j = 1}K(y_k - x_j)+ f_k(y)+B_k u \right) - L(y,\mu_N) - \gamma(u),
\end{split}
\end{align}
with $\mu_N = \frac{1}{N} \sum^N_{i = 1}\delta(x - x_i)$.
\end{theorem}
\end{framed}

\begin{remark} The general statement of the PMP contains both normal and abnormal minimizers. In our case, the simpler formulation of the PMP is given by the fact that we have normal minimizers only. This is a consequence of the fact that the final configuration is free, see e.g. \cite[Corollary 22.3]{clarke2013functional}.
\end{remark}

\begin{remark} \label{r-u}
The uniqueness of the maximizer of $\H_N$ follows from the same motivations reported in Remark \ref{r-uniquemax}. Indeed, the form of the Hamiltonian implies that for each $u^*\in \U$ it holds
\begin{align*}
u^* = \arg \max_{u \in \mathcal{U}} \H_N(y^*, q^*, x^*, p^*, u) \text{~~~~ when ~~~~} u^* = \arg \max_{u \in \mathcal{U}}\Pt{ \sum_{k=1}^m q^*_k \cdot B_k u - \gamma(u)}.
\end{align*}
In other terms, since the control acts on the $y$ variables only, then we have a simpler formulation for the maximization of the Hamiltonian $\H_N$.
\end{remark}

We now want to embed solutions of the PMP for Problem \ref{problemODE} as solutions of the extended PMP for Problem \ref{problemPDE}. As a first step, we prove that pairs control-trajectories $(u^*_N,(y^*_N,q^*_N,x^*_N,p^*_N))$ satisfying system \eqref{eq:PMPdiscrete} have support uniformly bounded in time and in $N \in \N$.

\begin{proposition}\label{p-boundedsupp}
Let $y^0\in \R^{dm}$, $\mu^0\in \Pc(\R^{d})$, and $\mu_N^0$ be as in Definition \ref{d-mfoc}--\ref{uno}. Let $u^*_N$ be a solution of Problem \ref{problemODE} with initial datum $(y^0,\mu^0_N)$, and let $(u^*_N,(y^*_N,q^*_N,x^*_N,p^*_N))$ be a pair control-trajectory satisfying the PMP for Problem \ref{problemODE} with initial datum $(y^0,\mu^0_N)$ and control $u^*_N$ given by Theorem \ref{PMPODE}.

Then the trajectories $(y^*_N(\cdot),q^*_N(\cdot),\nu^*_N(\cdot))$, where $\nu^*_N:= \Pi_N(x^*_N, p^*_N)$, are equibounded and equi-Lipschitz continuous from $[0,T]$ to $\Y$, where the space $\Y:=\R^{2dm}\times\PP(\R^{2d})$ is endowed with the distance
\bqn
\|(y,q,\nu)-(y',q',\nu')\|_{\Y}=\|y-y'\| +\|q-q'\|+\W_1(\nu,\nu').
\eqnl{e-Y}
Furthermore, there exists $R_T > 0$, depending only on $y^0, \supp(\mu^0), d, K, g, f_k, B_k, \mathcal{U},$ and $T$, such that $\supp(\nu^*_N(\cdot)) \subseteq B(0,R_T)$ for all $N \in \N$.
In particular, it holds $\H(y^*_N, q^*_N, \nu^*_N, u^*_N) = \H_c(y^*_N, q^*_N, \nu^*_N, u^*_N)$.
\end{proposition}
\begin{proof}
As a first step, notice that the pair $(y^*_N, x^*_N)$ solves the system \eqref{eq:discdyn}. It then follows from \eqref{discgronwall} and \eqref{discgronwalldot} that there exist two constants $\rho_T$ and $\rho'_T$, not depending on $N$ such that, for all $i=1,\dots,N$, for all $k=1,\dots, m$, and a.e.\ $t\in [0,T]$ we have
\begin{align}
\|y^*_{k,N}(t)\| \le \rho_T\,,\quad \|x^*_{i,N}(t)\| \le \rho_T \label{equibound}
\\
\|\dot y^*_{k,N}(t)\| \le \rho'_T\,,\quad \|\dot x^*_{i,N}(t)\| \le \rho'_T\,. \label{equilip}
\end{align}
It follows in particular that there exists a uniform constant $W_T$ such that
\begin{align*}
\left\|\frac1N\sum_{i=1}^N \omega(x^*_{i,N}(t))\right\| \le W_T
\end{align*}
for all $t\in [0,T]$.

We now observe that by an explicit computation
\begin{align}
\begin{split}\label{eq:nablaH}
N\nabla_{x_i}\H_N(y^*_N,q^*_N,x^*_N,p^*_N,u^*_N)\cdot e_l=&N \frac{r^*_{i,N}}{N} \cdot \frac{1}{N}\sum_{j=1}^N\Jac K(x^*_{i,N}-x^*_{j,N}) e_l-N\sum_{j=1}^N \frac{r^*_{j,N}}N \cdot \frac{1}{N}\Jac K(x^*_{j,N}-x^*_{i,N}) e_l\\
&+N\frac{r^*_{i,N}}{N}\cdot\nabla g(y^*_N)(x^*_{i,N}) -N\sum_{k=1}^mq^*_{k,N} \cdot \frac{1}{N} \Jac K(y^*_{k,N}-x^*_{i,N}) e_l\\
& -N\frac{1}{N}\Bigg(\nabla_{\xi}\ell\Pt{y^*_N,x^*_{i,N},\mbox{$\frac1N\sum_{j=1}^N\omega(x^*_{j,N})$}}\cdot e_l\\
&+\nabla_{\varsigma}\ell\Pt{y^*_N,x^*_{i,N},\mbox{$\frac1N\sum_{j=1}^N\omega(x^*_{j,N})$}} \Jac \omega(x^*_{i,N})\cdot e_l\Bigg)\\
=&\frac{1}{N}\sum_{j=1}^N (r^*_{i,N}-r^*_{j,N})\cdot\left(\Jac K(x^*_{i,N}-x^*_{j,N}) e_l\right)+r^*_{i,N}\cdot\Jac_x g(y^*_N)(x^*_{i,N}) \\
&-\sum_{k=1}^mq^*_{k,N}\cdot(\Jac K(y^*_{k,N}-x^*_{i,N}) e_l)-\nabla_{\xi}\ell\Pt{y^*_N,x^*_{i,N},\mbox{$\frac1N\sum_{j=1}^N\omega(x^*_{j,N})$}}\cdot e_l\\
&-\Pt{\nabla_{\varsigma}\ell\Pt{y^*_N,x^*_{i,N},\mbox{$\frac1N\sum_{j=1}^N\omega(x^*_{j,N})$}}\Jac\omega(x^*_{i,N})}\cdot e_l,
\end{split}
\end{align}
for each $i=1,\dots, N$ and each $l=1,\dots, d$ (where we have used that $\Jac K$ is even and we merged the first two terms). Therefore, since $\dot p^*_{i,N}$ solves \eqref{eq:PMPdiscrete} and $\dot r^*_{i,N} = N \dot p^*_{i,N}$, we get
\begin{align*}
-\dot r^*_{i,N}(t)\cdot e_l=&\frac{1}{N}\sum_{j=1}^N(r^*_{i,N}(t)-r^*_{j,N}(t))\cdot(\Jac K(x^*_{i,N}(t)-x^*_{j,N}(t)) e_l)\\
&+r^*_{i,N}(t)\cdot (\Jac_y g(y^*_{N}(t))(x^*_{i,N}(t)) e_l)- \sum_{k=1}^m q^*_{k,N}(t) \cdot (\Jac K(y^*_{k,N}(t)-x^*_{i,N}(t)) e_l)\\
&-\nabla_{\xi}\ell\Pt{y^*_{N}(t),x^*_{i,N}(t),\mbox{$\frac1N\sum_{j=1}^N\omega(x^*_{j,N}(t))$}}\cdot e_l\\
&-\left(\nabla_{\varsigma}\ell\Pt{y^*_{N}(t),x^*_{i,N}(t),\mbox{$\frac1N\sum_{j=1}^N\omega(x^*_{j,N}(t))$}}\Jac\omega(x^*_{i,N}(t))\right)\cdot e_l,
\end{align*}
for each $i=1,\dots, N$, each $l=1,\dots, d$, and a.e.\ $t\in [0,T]$, where we have used the fact that $\Jac K$ is even.
We now denote with $L_T$ a uniform constant such that
\begin{eqnarray*}
&\displaystyle
\|\Jac K\|_{L^\infty(B(0,\rho_T),\R^{d \times d})}\le L_T\,,\quad \sup_{|y|\le\sqrt{m}R_T}\|\Jac_y g(y)(\cdot)\|_{L^\infty(B(0,\rho_T),\R^{d \times d})}\le L_T\\
&\displaystyle
\|\nabla_{\xi}\ell\|_{L^\infty(B(0,\sqrt{m}\rho_T)\times B(0,\rho_T)\times B(0,W_T),\R^d)}\le L_T\,,\quad\|\nabla_{\varsigma}\ell\|_{L^\infty(B(0,\sqrt{m}\rho_T)\times B(0,\rho_T)\times B(0,W_T),\R^d)}\le L_T\,,\\
&\displaystyle
\|\Jac\omega\|_{L^\infty(B(0,\rho_T),\R^{d \times d})}\le L_T\,,
\end{eqnarray*}
and we easily get the estimate
\begin{equation}\label{erre}
\|\dot r^*_{i,N}(t)\|\le \sqrt{d}L_T\left(2\|r^*_{i,N}(t)\|+\frac{1}{N}\sum_{j=1}^N\|r^*_{j,N}(t)\|+\sum_{k=1}^m\|q^*_{k,N}(t)\|+1+L_T\right)
\end{equation}
for each $i=1,\dots, N$ and a.e.\ $t\in [0,T]$. An explicit computation of $\nabla_{y_k}\H_N$ and a similar argument, possibly with another constant $L_T$, show the estimate
\begin{equation}\label{qu}
\|\dot q^*_{k,N}(t)\|\le \sqrt{d}L_T\left(\frac{1}{N}\sum_{i=1}^N\|r^*_{i,N}(t)\|+2\|q^*_{k,N}(t)\|+L_T\right)
\end{equation}
for each $k=1,\dots, m$ and a.e.\ $t\in [0,T]$. 
We now set
$$
\eps_N(t):=\sum_{k=1}^m \|q^*_{k,N}(t)\|+\frac1N\sum_{i=1}^N \|r^*_{i,N}(t)\|\,,
$$
and observe that it holds
$$
|\dot \eps_N(t)|\le \sum_{k=1}^m \|\dot q^*_{k,N}(t)\|+\frac1N\sum_{i=1}^N \|\dot r^*_{i,N}(t)\|\,,
$$
therefore \eqref{erre} and \eqref{qu} yield
\begin{equation}\label{eps}
|\dot \eps_N(t)|\le \sqrt{d}L_T\left(4\eps_N(t)+1+2L_T\right)\,.
\end{equation}
Defining then the increasing functions $\eta_N(t)$ through $\eta_N(t):=\sup_{\tau\in[0,t]}\eps_N(T-\tau)$, and observing that it holds $\eta_N(0)=0$ for the boundary conditions in Theorem \ref{PMPODE}, from \eqref{eps} and Gronwall's Lemma we obtain
$$
\eta_N(\tau)\le\sqrt{d}L_T\tau(1+2L_T)e^{(4\sqrt{d}L_T)\tau}
$$
and with this
\begin{equation}\label{eps-equibound}
\eps_N(t)\le \eta_N(T)\le\sqrt{d}L_TT(1+2L_T)e^{(4\sqrt{d}L_T)T}:=C_T
\end{equation}
for all $t\in [0,T]$.
Plugging into \eqref{eps}, we get the existence of a constant $C'_T$ such that
\begin{equation}\label{eps-equilip}
|\dot \eps_N(t)|\le C'_T
\end{equation}
for a.e.\ $t\in [0,T]$. Since by definition of $\nu^*_N(t)$ and standard properties of the Wasserstein distance $\mathcal W_1$ it holds
$$
\W_1(\nu^*_N(t+\tau),\nu^*_N(t))\le\sqrt{2}\left(\frac1N \sum_{i=1}^N \|x^*_{i,N}(t+\tau)-x^*_{i,N}(t)\|+\frac1N\sum_{i=1}^N \|r^*_{i,N}(t+\tau)-r^*_{i, N}(t))\|\right)\,,
$$
from the previous inequality, \eqref{equibound}, \eqref{equilip}, \eqref{eps}, \eqref{eps-equibound}, and \eqref{eps-equilip} we obtain that $y^*_N(t)$ and $q^*_N(t)$ are equibounded, that there exist a constant, denoted by $R_T$, such that ${\rm supp}(\nu^*_N(t))\subset B(0, R_T)$ for all $t\in [0,T]$ and that $(y^*_N,q^*_N,\nu^*_N)$ are equi-Lipschitz continuous from $[0,T]$ with values in $\Y$.
\end{proof}

\bp
\label{p-embed} Let $N \in \N$ and $u^*_N\in L^p([0,T];\U)$ be an optimal control for Problem \ref{problemODE} with given initial datum $(y^0_N, x^0_N)\in\R^{dm+dN}$, and $(y^*_N(\cdot), q^*_N(\cdot), x^*_N(\cdot), p^*_N(\cdot))\in \mathrm{Lip}([0,T],\R^{2dm+2dN})$ a corresponding trajectory of the PMP with maximized Hamiltonian $\H_N$.

Define $\nu_N^*:=\Pi_N(x^*_{1,N},p^*_{1,N}, \ldots,x^*_{N,N},p^*_{N,N})$ with $\Pi_N$ as in \r{e-PiNint}, and assume that $\supp(\nu^*_N(\cdot)) \subseteq B(0,R_T)$. Then, the control $u^*_N$ is optimal for Problem \ref{problemPDE} and $(y^*_N,q^*_N,\nu^*_N,u^*_N)$ satisfies the extended Pontryagin Maximum Principle.
\ep
\bproof First observe that, by Proposition \ref{p-boundedsupp}, $\H_c(y^*_N,q^*_N,\nu^*_N,u^*_N) = \H(y^*_N,q^*_N,\nu^*_N,u^*_N)$ and that for every $t \in [0,T]$
$$u^*_N(t)= \arg\max_{u \in \mathcal{U}} \H_N(y^*_N(t),q^*_N(t),x_N^*(t),p_N^*(t),u)\iff u^*_N(t)= \arg\max_{u \in \mathcal{U}} \H(y^*_N(t),q^*_N(t),\nu^*_N(t),u),$$ due to the specific form of the Hamiltonian $\H_N$ and $\H$, see Remark \ref{r-u}.

We now prove that
\begin{align*}
\begin{cases}
\dot y^*_{k,N} &=\nabla_{q_k}\H_N(y^*_N,q^*_N,x^*_N,p^*_N,u^*_N), \\
\dot q^*_{k,N} &=-\nabla_{y_k} \H_N(y^*_N,q^*_N,x^*_N,p^*_N,u^*_N),
\end{cases} \quad \Longrightarrow \quad
\begin{cases}
\dot y^*_{k,N} &=\nabla_{q_k}\H(y^*_N,q^*_N,\nu^*_N,u^*_N), \\
\dot q^*_{k,N} &=-\nabla_{y_k} \H(y^*_N,q^*_N,\nu^*_N,u^*_N),
\end{cases}
\end{align*}
i.e., that if the $(y,q)$ variables satisfy the PMP for Problem \ref{problemODE} then they satisfy the extended PMP for Problem \ref{problemPDE}. It is sufficient to observe that $\H_N$ can be rewritten in terms of $\nu^*_N(\cdot)$ as follows
\begin{align*}
\begin{split}
\H_N(y^*_N,q^*_N,\nu^*_N,u^*_N)=&\int_{\R^{4d}} r \cdot (K\star \pi_{1\#}\nu^*_N)(x)\,d\nu^*_N(x,r)+\int_{\R^{2d}} r \cdot g(y^*_N)(x)\,d\nu^*_N(x,r)\\
& +\sum_{k=1}^m\int_{\R^{2d}} q^*_{k,N} \cdot K(y^*_{k,N}-x)\,d\nu^*_N(x,r)+\sum_{k=1}^m q^*_{k,N} \cdot (f_k(y^*_{k,N})+B_k u^*_N)\\
& -L(y^*_N,\pi_{1\#}\nu^*_N)-\gamma(u^*_N),
\end{split}
\end{align*}
where we used the variable $r=Np$. Comparing it with $\H(y^*_N,q^*_N,\nu^*_N,u^*_N)$, one has that their expressions coincide up to the first term. 
Since such first term is independent on $y_k$ and $q_k$, then $\nabla_{y_k} \H_N=\nabla_{y_k} \H$ and $\nabla_{q_k} \H_N=\nabla_{q_k} \H$, hence equations for $\dot y^*_{k,N},\dot q^*_{k,N}$ in the PMP for Problem \ref{problemODE} and in the extended PMP for Problem \ref{problemPDE} coincide.

We now prove a similar result for the $(x^*_{i,N},r^*_{i,N})$ variables, with $r^*_{i,N}=N p_{i,N}^*$. After this change of variable, the third and the fourth equation in \eqref{eq:PMPdiscrete} become
$$
\begin{cases}
\dot{x}^*_{i, N} & = N\nabla_{r_i} \mathbb{H}_N(y^*_N, q^*_N, x^*_N, p^*_N,u^*_N) \\
\dot{r}^*_{i, N} & = - N\nabla_{x_i} \mathbb{H}_N(y^*_N, q^*_N, x^*_N, p^*_N,u^*_N)\,.
\end{cases}
$$
We want to prove that the following identity holds
\bqn
J(\nabla_\nu \H_c(y^*_N,q^*_N,\nu^*_N,u^*_N))(x^*_{i,N},r^*_{i,N})=\Pt{\ba{c}N\nabla_{r_i} \H_N(y^*_N,q^*_N,x^*_N,p^*_N,u^*_N)\\-N\nabla_{x_i}\H_N(y^*_N,q^*_N,x^*_N,p^*_N,u^*_N)\ea},
\eqnl{e-uguale}
i.e., that the Hamiltonian vector fields generated by $\H$ and $\H_N$ coincide in each point $(x^*_{i,N},r^*_{i,N})$. The presence of the constant $N$ in the right-hand side is due to the change of variables $r^*_{i,N}=N p_{i,N}^*$. By applying $J^{-1}$ on both sides of \r{e-uguale}, we need to prove
\bqn
\!\!\!\!\!\!\!\!\!\!\!\nabla_\nu \H_c(y^*_N,q^*_N,\nu^*_N,u^*_N)(x^*_{i,N},r^*_{i,N})\cdot e_l &\!\!\!\!\!\!\!\!\!=N\nabla_{x_i}\H_N(y^*_N,q^*_N,x^*_N,p^*_N,u^*_N)\cdot e_l &\mbox{ for $l=1,\ldots,d,$~~~}\label{e-embed}\\
\!\!\!\!\!\!\!\!\!\!\!\nabla_\nu \H_c(y^*_N,q^*_N,\nu^*_N,u^*_N)(x^*_{i,N},r^*_{i,N})\cdot e_l &\!\!\!\!=N\nabla_{r_i}\H_N(y^*_N,q^*_N,x^*_N,p^*_N,u^*_N)\cdot e_{l-d} &\mbox{~for $l=d+1,\ldots,2d.$~~~}\label{e-embed2}
\eqn
By writing explicitly the left hand sides of \r{e-embed} and \r{e-embed2} by using the expressions \r{e-gradx}-\r{e-gradr} and evaluating them in $(x^*_{i,N},r^*_{i,N})$, we have
\begin{align*}
\begin{split}
\nabla_\nu \H_c(y^*_N,q^*_N,\nu^*_N,u^*_N)(x^*_{i,N},r^*_{i,N})\cdot e_l&=\frac{1}{N}\sum_{j=1}^N (r^*_{i,N}-r^*_{j,N}) \cdot \left(\Jac K(x^*_{i,N}-x^*_{j,N}) e_l\right)+ r^*_{i,N} \cdot \left(\Jac_x g(y^*_N)(x^*_{i,N}) e_l\right)\\
&-\sum_{k=1}^m q^*_{k,N} \cdot \left(\Jac K(y^*_{k,N}-x^*_{i,N}) e_l\right)-\nabla_{\xi}\ell\Pt{y^*_N,x^*_{i,N},\mbox{$\frac1N\sum_{j=1}^N\omega(x^*_{j,N})$}}\cdot e_l\\
&-\left(\nabla_{\varsigma}\ell\Pt{y^*_N,x^*_{i,N},\mbox{$\frac1N\sum_{j=1}^N\omega(x^*_{j,N})$}}\Jac\omega(x^*_{i,N})\right)\cdot e_l,
\end{split}
\end{align*}
for $l=1,\ldots,d$, so that \eqref{e-embed} follows immediately from \eqref{eq:nablaH}. Similarly, we have
\bqn
\nabla_\nu \H_c(y^*_N,q^*_N,\nu^*_N,u^*_N)(x^*_{i,N},r^*_{i,N})\cdot e_l&=&\frac{1}{N}\sum_{j=1}^N K(x^*_{i,N}-x^*_{j,N})\cdot e_{l-d} +g(y^*_N)(x^*_{i,N})\cdot e_{l-d}
\eqnn
for $l=d+1,\ldots,2d$, which coincides with the right hand side of \r{e-embed2} by an explicit computation.
Since the boundary conditions of Problem \ref{problemODE} and Problem \ref{problemPDE} coincide too, after the identification $\nu_N^*:=\Pi_N(x^*_{1,N},p^*_{1,N}, \ldots,x^*_{N,N},p^*_{N,N})$, the result follows now by \eqref{e-embed}-\eqref{e-embed2} arguing, for instance, as in \cite[Lemma 4.3]{MFOC}.
\eproof
\brem\label{rem3}
It is interesting to observe that the embedding of a trajectory of the PMP for Problem \ref{problemODE} to the empirical measure formulation depends on the number $N$ of agents, see the definition of $\Pi_N$ in \r{e-PiNint}. This is a consequence of the fact that the Hamiltonian of the PMP for Problem \ref{problemODE} actually depends on the number of agents. Indeed, consider a population composed of a unique agent $(x_1,p_1)$, for which the first term of the Hamiltonian reads as $p_1 \cdot g(y)(x_1)$. Consider now a population composed of two agents $(x_1,p_1,x_2,p_2)$ satisfying $x_1=x_2$ and $p_1=p_2$, for which the first term of the Hamiltonian reads as $2p_1\cdot g(y)(x_1)$.

Clearly, in both cases the empirical measure in the state variables is $\mu_1=\mu_2=\delta(x-x_1)$, while the definition of $\Pi_N$ gives two different empirical measures for the cotangent bundle: $\nu_1=\delta(x-x_1,r-p_1)$ and $\nu_2=\frac{1}{2}\delta(x-x_1,r-2p_1)$. This difference is needed to compensate the dependence of the Hamiltonian of the PMP for Problem \ref{problemODE} on the number $N$ of agents.
\erem

\section{The Wasserstein gradient}
\label{s-wass}

We anticipated in Section \ref{s:intro} that the dynamics of $\nu^*$ in \r{e-PMP} is an Hamiltonian flow in the Wasserstein space of probability measures, in the sense of \cite{ambrosio}. This means that the vector field $\nabla_{\nu} \mathbb{H}_c(\nu^*)$ is an element with minimal norm in the Fr\'echet subdifferential at the point $\nu^*$ of the maximized Hamiltonian $\mathbb{H}_c$ introduced in Theorem \ref{t-main} (we drop for simplicity the $y$, $q$ and $u$ dependency). The proof of this fact shall follow the strategy adopted to obtain analogous results in \cite[Chapter 10]{AGS}, which however cannot be applied verbatim to our case due to the peculiar nature of our operators. In order to use those techniques, we consider our functionals defined on $\mathcal{P}_2(\R^{2d})$ instead than on $\PP(\R^{2d})$. Since we have already proven in Proposition \ref{p-boundedsupp} that, whenever we start from a compactly supported initial datum, the dynamics remains compactly supported uniformly in time, this 
assumption does not alter our conclusions.

We start with some basic definitions and general results on functionals defined on $\mathcal{P}_2(\R^{2d})$: the following one is motivated by Definition 10.3.1 and Remark 10.3.3 in \cite{AGS}.

\bdeff
Let $\psi:\mathcal{P}_2(\R^{2d}) \to (-\infty,+\infty]$ be a proper and lower semicontinuous functional, and let $\nu_0 \in D(\psi)$. We say that $w \in L^2_{\nu_0}(\R^{2d})$ belongs to the {\it(Fr\'{e}chet) subdifferential of $\psi$ at $\nu_0$}, in symbols $w \in \partial \psi(\nu_0)$ if and only if for any $\nu_1 \in \mathcal{P}_2(\R^{2d})$ it holds
\bqn
\psi(\nu_1) - \psi(\nu_0) \geq \inf_{\rho \in \Gamma_o(\nu_0,\nu_1)}\int_{\R^{4d}} w(z_0)\cdot( z_1 - z_0) d\rho(z_0,z_1) + o(\W_2(\nu_1, \nu_0)).
\eqnn
\edeff


\bp[\cite{AGS}, Theorem 10.3.10] \label{t-minnorm}
Fix the functional $\psi:\mathcal{P}_2(\R^{2d}) \to (-\infty,+\infty]$. Then, for every $\nu_0 \in D(\psi)$, the {\it metric slope}
\bqn
|\partial \psi|(\nu_0) = \limsup_{\nu_1 \to \nu_0} \frac{(\psi(\nu_1) - \psi(\nu_0))^+}{\W_2(\nu_1, \nu_0)}
\eqnn
satisfies
\bqn \label{e-estAbNorm}
|\partial \psi|(\nu_0) \leq \|w\|_{L^2_{\nu_0}}
\eqn
for every $w \in \partial \psi(\nu_0)$.
\ep

The following property shall guarantee that the subdifferential of $\H_c$ is nonempty.

\bdeff
A proper, lower semicontinuous functional $\psi: \mathcal{P}_2(\R^n) \to (-\infty,+\infty]$ is {\it semiconvex along geodesics} whenever, for every $\nu_0, \nu_1 \in \mathcal{P}_2(\R^n)$ and $\rho \in \Gamma_o(\nu_0, \nu_1)$ there exists $C \in \R$ for which it holds
\bqn
\psi(((1-s)\pi_1 + s\pi_2)_{\#}\rho) \leq (1-s) \psi(\nu_0) + s \psi(\nu_1) + Cs(1-s) \W^2_2(\nu_0, \nu_1)) \text{ for every } s \in [0,1].
\eqnn
\edeff

In what follows, we shall fix $y, q \in \R^{dm}$ and $u \in L^1([0,T];\U)$ and we write, for the sake of compactness, $\H_c(\nu)$ in place of $\H_c(y,q,\nu,u)$. Moreover,  $\mathcal{K}$ shall denote a convex, compact subset of $\R^{2d}$ and $z = (x,r)$ a variable in $\R^{2d}$.

Whenever $\supp(\nu) \subseteq \overline{B(0,R_T)}$, $\mathbb{H}_c(\nu)$ can be rewritten as
\bqn
\mathbb{H}_c(\nu) = \frac{1}{2}\int_{\R^{4d}} \mathcal{F}(z - z') d \nu(z) \nu(z') + \int_{\R^{d}} \mathcal{G}(z) d \nu(z) - \int_{\R^d} \ell(\pi_1(z), \mathsmaller \int \om \pi_{1\#}\nu) d\nu(z) + Q,
\eqnn
where
\begin{align*}
\mathcal{F}(x,r) &= r \cdot K(x) \\
\mathcal{G}(x,r) &= r \cdot g(y)(x) + \sum^m_{k = 1} q_k \cdot K(y_k - x),
\end{align*}
and $Q$ collects all the remaining terms not depending on $\nu$. Notice that $\mathcal{F}$ is an even function.

In order to prove the semiconvexity of $\H_c$, we shall establish the semiconvexity of the following functionals:
\begin{align*}
\hat{\H}^1_c(\nu) & = \frac{1}{2}\int_{\R^{4d}} \hat{\mathcal{F}}(z - z') d \nu(z) \nu(z') + \int_{\R^{d}} \hat{\mathcal{G}}(z) d \nu(z), \\
\hat{\H}^2_c(\nu) & = \int_{\R^d} \hat{\ell}(z, \mathsmaller \int \hat{\om} \nu) d\nu(z),
\end{align*}
where $\hat{\mathcal{F}}$, $\hat{\mathcal{G}}$, $\hat{\ell}$, and $\hat{\omega}$ are $\mathcal{C}^2$ functions. The desired result will then follow by noticing that $\H_c(\nu) = \hat \H_c^1(\nu) + \hat \H_c^2(\nu)$ for $\hat{\mathcal{F}} = \mathcal{F}$, $\hat{\mathcal{G}} = \mathcal{G}$, $\hat{\ell} = -\ell \circ (\pi_1, \textup{Id})$, $\hat{\omega} = \omega \circ \pi_1$ and $\mathcal{K} = \overline{B(0,R_T)}$.

The following simple property will be needed to prove semiconvexity of the above functionals.

\begin{lemma} \label{r-suppconv}
Let $\nu_0,\nu_1 \in \Pc(\R^{2d})$ with support contained in $\mathcal K$. Let $\rho \in \Gamma(\nu_0,\nu_1)$ and set
\begin{align} \label{e-measconv}
\nu_s = ((1-s)\pi_1 + s\pi_2)_{\#} \rho,
\end{align}
for every $s \in [0,1]$. Then, it holds
$$\supp(\nu_s) \subseteq \mathcal K \quad \text{ for all } s \in [0,1].$$
\end{lemma}
\begin{proof}
We first notice, that for every $\rho \in \Gamma(\nu_0,\nu_1)$ it holds
\begin{align}\label{r-suppK}
\supp(\rho) \subseteq \mathcal K \times \mathcal K\,.
\end{align}
This follows from the equality
\begin{align*}
\R^{4d} \backslash (\mathcal K \times \mathcal K) = (\R^{2d} \times (\R^{2d} \backslash \mathcal K)) \cup ((\R^{2d} \backslash \mathcal K) \times \R^{2d})
\end{align*}
and from the fact that both $\R^{2d} \times (\R^{2d} \backslash \mathcal K))$ and $(\R^{2d} \backslash \mathcal K) \times \R^{2d}$ are $\rho$-null sets by hypothesis. 

To prove the Lemma, it suffices to show that for all $f \in \mathcal{C}(\R^{2d})$ satisfying $f \equiv 0$ on $\mathcal K$ it holds
\begin{align}\label{e-Kdes}
\int_{\R^{2d}} f d\nu_s = 0.
\end{align}
Indeed,
\begin{align*}
\int_{\R^{2d}} f d\nu_s &= \int_{\R^{4d}} f d((1-s)\pi_1 + s\pi_2)_{\#} \rho(z_0,z_1) \\
&=\int_{\R^{4d}} f((1-s)z_0 + s z_1) d \rho(z_0,z_1) \\
&=\int_{\mathcal K \times \mathcal K} f((1-s)z_0 + s z_1) d \rho(z_0,z_1),
\end{align*}
since, by \eqref{r-suppK}, $\supp(\rho) \subseteq \mathcal K \times \mathcal K$. From the convexity of $\mathcal K$ follows that $(1-s)z_0 + sz_1 \in \mathcal K$ for every $s \in [0,1]$, which, together with the assumption $f \equiv 0$ in $\mathcal K$, yield \eqref{e-Kdes}, as desired.
\end{proof}

In what follows, we shall make use of the following, well-known result.

\begin{remark}
Let $\mathcal K$ be a convex, compact subset of $\R^{2d}$ and let $f \in \mathcal{C}^2(\R^{2d};\R)$. Then there exists $C_{\mathcal K,f} \in \R$ depending only on $\mathcal K$ and $f$ such that
\begin{align}\label{e-c2conv}
f((1-s)x_0 + sx_1) \leq (1-s)f(x_0) + s f(x_1) + C_{\mathcal K,f} s(1-s) \|x_0 - x_1\|^2,
\end{align}
for every $x_0,x_1 \in \R^{2d}$ and $s \in [0,1]$.
\end{remark}

We now prove the semiconvexity of $\hat{\H}^1_c$.
\begin{lemma}\label{l-semiconvH1}
Let $\nu_0,\nu_1 \in \Pc(\R^{2d})$ and let $\rho \in \Gamma(\nu_0,\nu_1)$. Then, there exists $C \in \R$ independent of $\nu_0$ and $\nu_1$ for which
\begin{align*}
\hat{\H}^1_c(((1-s)\pi_1 + s\pi_2)_{\#}\rho) \leq (1-s) \hat{\H}^1_c(\nu_0) + s \hat{\H}^1_c(\nu_1) + Cs(1-s) \W^2_2(\nu_0, \nu_1)
\end{align*}
holds for every $s \in [0,1]$.
\end{lemma}
\begin{proof}
We may assume $\supp(\nu_0),\supp(\nu_1) \subseteq \mathcal K$ for some convex and compact set $\mathcal K \subset \R^{2d}$, otherwise the inequality is trivial. Hence, from Lemma \ref{r-suppconv}, it follows $\supp(\nu_s) \subseteq \mathcal K$ for every $s \in [0,1]$. But then, since $\hat{\mathcal{F}}$ and $\hat{\mathcal{G}}$ are both $\mathcal{C}^2$, the result follows as in \cite[Proposition 9.3.2, Proposition 9.3.5]{AGS}.
\end{proof}

\begin{corollary}
Let $\hat{\omega} \in \mathcal{C}^2(\R^{2d};\R^{d})$, $\nu_0,\nu_1 \in \Pc(\R^{2d})$, $\rho \in \Gamma(\nu_0,\nu_1)$ and define $\nu_s$ as in \eqref{e-measconv} for $s \in [0,1]$. If we set
\begin{align}\label{e-intom}
\xi_s = \int_{\R^{2d}} \hat{\omega} d \nu_s,
\end{align}
then
\begin{align*}
\|\xi_s - (1-s)\xi_0 - s\xi_1\| \leq C s (1-s) \W^2_2(\nu_0,\nu_1),
\end{align*}
for all $s\in [0,1]$, where $C$ is independent of $\nu_0$ and $\nu_1$.
\end{corollary}
\begin{proof}
Follows from Lemma \ref{l-semiconvH1} applied first to the functions $\hat{\mathcal{F}} \equiv 0$ and $\hat{\mathcal{G}} \equiv \hat{\omega}$, and then to $\hat{\mathcal{F}} \equiv 0$ and $\hat{\mathcal{G}} \equiv -\hat{\omega}$.
\end{proof}

The semiconvexity of $\hat{\H}^2_c$ will be deduced as a corollary of the following estimate.

\begin{lemma}\label{l-semiconvl}
Suppose that $\hat{\ell} \in \mathcal{C}^2(\R^{2d}\times\R^d;\R)$, let $z_0,z_1 \in \mathcal K$ and set $z_s = (1-s)z_0 + sz_1$ for all $s \in [0,1]$. Furthermore, let $\nu_0,\nu_1 \in \Pc(\R^{2d})$, $\rho \in \Gamma(\nu_0,\nu_1)$ and define $\nu_s$ and $\xi_s$ as in \eqref{e-measconv} and \eqref{e-intom} for $s \in [0,1]$. Then, for all $s \in [0,1]$, it holds
\begin{align*}
\hat{\ell}(z_s,\xi_s) \leq (1-s) \hat{\ell}(z_0,\xi_0) + s \hat{\ell}(z_1,\xi_1) +  C_{\mathcal K,\hat{\ell},\hat{\omega}} s(1-s) \W^2_2(\nu_0,\nu_1) + C_{\mathcal K,\hat{\ell},\hat{\omega}} s(1-s) \|z_0 - z_1\|^2,
\end{align*}
for some constant $C_{\mathcal K,\hat{\ell},\hat{\omega}}$ depending only on $\mathcal K,\hat{\ell}$ and $\hat{\omega}$.
\end{lemma}
\begin{proof}
Since $\mathcal K$ is compact, $z_s \in \mathcal K$ for all $s \in [0,1]$. Moreover, $(1-s)\xi_0 + s\xi_1 \in \mathcal K'$ for all $s \in [0,1]$, for some convex and compact set $\mathcal K' \subset \R^d$. Notice that from \eqref{e-c2conv} follows 
\begin{align}
\begin{split}\label{a1}
\hat{\ell}(z_s,(1-s)\xi_0 + s\xi_1) \leq (1-s) \hat{\ell}(z_0,\xi_0) + s \hat{\ell}(z_1,\xi_1) + C_{\mathcal K,\mathcal K'}s(1-s)\left( \|z_0 - z_1\|^2 + \|\xi_0 - \xi_1\|^2 \right),
\end{split}
\end{align}
and from the definition of $\xi_s$ and Jensen's inequality, we get
\begin{align}\label{a2}
\|\xi_0 - \xi_1\|^2 &\leq \Lip_{\mathcal K}(\omega) \W^2_1(\nu_0,\nu_1) \leq \Lip_{\mathcal K}(\omega) \W^2_2(\nu_0,\nu_1).
\end{align}
Moreover, for every $s\in[0,1]$ it holds
\begin{align}
\begin{split}\label{a3}
\|\hat{\ell}(z_s,\xi_s) - \hat{\ell}(z_s,(1-s)\xi_0 + s\xi_1)\| & \leq \Lip_{\mathcal K \times \mathcal K'} \|\xi_s - (1-s)\xi_0 - s\xi_1\| \\
& \leq \Lip_{\mathcal K \times \mathcal K'} s(1-s) C \W^2_2(\nu_0,\nu_1).
\end{split}
\end{align}
Hence, for every $s\in[0,1]$, using \eqref{a1}, \eqref{a2} and \eqref{a3}, we get
\begin{align*}
\hat{\ell}(z_s,\xi_s) & = \hat{\ell}(z_s,\xi_s) - \hat{\ell}(z_s,(1-s)\xi_0 + s\xi_1) + \hat{\ell}(z_s,(1-s)\xi_0 + s\xi_1) \\
& \leq (1-s) \hat{\ell}(z_0,\xi_0) + s \hat{\ell}(z_1,\xi_1) +  C_{\mathcal K,\hat{\ell},\hat{\omega}} s(1-s) \W^2_2(\nu_0,\nu_1) + C_{\mathcal K,\hat{\ell},\hat{\omega}} s(1-s) \|z_0 - z_1\|^2.
\end{align*}
\end{proof}

\begin{corollary}\label{c-semiconvH2}
Let $\nu_0,\nu_1 \in \Pc(\R^{2d})$ and $\rho \in \Gamma_o(\nu_0,\nu_1)$. Then, there exists $C \in \R$ independent of $\nu_0$ and $\nu_1$ for which
\begin{align*}
\hat{\H}^2_c(((1-s)\pi_1 + s\pi_2)_{\#}\rho) \leq (1-s) \hat{\H}^2_c(\nu_0) + s \hat{\H}^2_c(\nu_1) + Cs(1-s) \W^2_2(\nu_0, \nu_1)
\end{align*}
holds for every $s \in [0,1]$.
\end{corollary}
\begin{proof}
Notice that, by Lemma \ref{r-suppconv}, $\hat{H}^2_c(\nu_s)$ can be rewritten as
\begin{align*}
\hat{\H}^2_c(\nu_s) = \int_{\mathcal{K} \times \mathcal{K}} \hat{\ell}(z_s,\xi_s)d\rho(z_0,z_1),
\end{align*}
Furthermore, since $\rho \in \Gamma_o(\nu_0,\nu_1)$ it holds
\begin{align*}
\int_{\mathcal{K} \times \mathcal{K}} \|z_0 - z_1\|^2 d\rho(z_0,z_1) = \int_{\R^{4d}} \|z_0 - z_1\|^2 d\rho(z_0,z_1) = \W^2_2(\nu_0,\nu_1),
\end{align*}
the thesis follows from Lemma \ref{l-semiconvl}.
\end{proof}

\begin{proposition}\label{l-semiconvex}
The functional $\H_c$ is semiconvex along geodesics.
\end{proposition}
\begin{proof}
Follows directly from Lemma \ref{l-semiconvH1} and Corollary \ref{c-semiconvH2}, by noticing that $\H_c(\nu) = \hat \H_c^1(\nu) + \hat \H_c^2(\nu)$ for $\hat{\mathcal{F}} = \mathcal{F}$, $\hat{\mathcal{G}} = \mathcal{G}$, $\hat{\ell} = -\ell \circ (\pi_1, \textup{Id})$, $\hat{\omega} = \omega \circ \pi_1$ and $\mathcal{K} = \overline{B(0,R_T)}$.
\end{proof}


We define the vector field $\nabla_\nu \mathcal{L}:\R^{2d} \rightarrow \R^{2d}$ as
\bqn
\nabla_{\nu} \mathcal{L}(z) = \left[\begin{matrix}
\nabla_{\xi} \ell(y, \pi_1(z), \mathsmaller \int \om \pi_{1\#}\nu) +  \nabla_{\varsigma} \ell(y, \pi_1(z), \mathsmaller \int \om \pi_{1\#}\nu) \Jac \om(\pi_1(z)) \\
0 \\
\end{matrix} \right],
\eqnn
for every $z \in \R^{2d}$. This notation is reminiscent of the fact that this vector field will eventually turn out to be the 2-Wasserstein gradient of the functional $L$, as it will follow from Theorem \ref{t-wassgrad} in the case $\mathcal{F} \equiv \mathcal{G} \equiv 0$. We can thus define our candidate vector field for the Wasserstein gradient $\nabla_{\nu} \H_c(\nu_0)$ in the case that $\supp(\nu_0) \subseteq B(0,R_T)$:
\bqn \label{e-flow}
w := (\nabla \mathcal{F}) \star \nu + \nabla \mathcal{G} - \nabla_{\nu} \mathcal{L}.
\eqn
Notice that, by Hypotheses (H), $w$ is a continuous function in $z$, and hence it is well-defined $\nu$-a.e..

\bl
Let $\nu \in \mathcal{P}_c(\R^{2d})$. Then $w$ defined by \eqref{e-flow} belongs to $L^p_{\nu}(\R^{2d})$ for every $p \in [1,+\infty]$, and it satisfies
\begin{align} \label{e-intPart}
\int_{\R^{4d}}  w(z_0)\cdot( z_1 - z_0) d\rho(z_0,z_1) = \!\! \int_{\R^{6d}} \left( \nabla \mathcal{F}(z_0 - z_2) + \nabla \mathcal{G}(z_0) - \nabla_{\nu} \mathcal{L}(z_0)\right)\cdot (z_1 - z_0) d\rho(z_0,z_1) d\nu(z_2)
\end{align}
for every plan $\rho \in \Gamma(\nu,\nu')$ such that $\nu' \in \mathcal{P}_c(\R^{2d})$.
\el
\bproof
Since $w$ is continuous, the fact that $w$ is $L^p_{\nu}$-integrable follows the fact that $\nu$ has compact support. Equation \r{e-intPart} then follows by Fubini-Tonelli and from the fact that $\rho$ is compactly supported too by Remark \ref{r-suppconv}.
\eproof

\begin{theorem}\label{t-wassgrad}
Let $\nu \in \mathcal{P}_2(\R^{2d})$ be such that $\supp(\nu) \subseteq B(0,R_T)$. Then $\nu \in D(|\partial \H_c|)$ if and only if $w$ as in \r{e-flow} belongs to $L^2_{\nu}(\R^{2d})$. In this case, $\|w\|_{L^2_{\nu}} = |\partial \H_c|(\nu)$, i.e., $w$ is an element with minimal norm in $\partial \H_c(\nu)$.
\end{theorem}
\bproof
We start by assuming that $\nu \in \mathcal{P}_2(\R^{2d})$ satisfies $|\partial \H_c|(\nu) < +\infty$ and proving that this implies that $w$ belongs to $L^2_{\nu}(\R^{2d})$ and that $\|w\|_{L^2_{\nu}} \leq |\partial \H_c|(\nu)$.
We compute the directional derivative of $\H_c$ along a direction induced by the transport map $Id + \xi$, where $\xi$ is a smooth function with compact support such that $\supp((Id + s\xi)_{\#}\nu) \subseteq \overline{B(0,R_T)}$ for any sufficiently small $s > 0$. If we denote with
\bqn
\mathcal{L}_1(s) & = & \ell(y,\pi_1(z_0)+ s(\pi_1 \circ \xi)(z_1), \mathsmaller \int \om d(\pi_1 \circ (Id + s\xi))_{\#}\nu), \nonumber \\
\mathcal{L}_2(s) & = & \ell(y,\pi_1(z_0), \mathsmaller \int \om d(\pi_1 \circ (Id + s\xi))_{\#}\nu),
\eqnn
then the map
\bqn
s &\mapsto& \frac{\mathcal{F}((z_0 - z_1) + s(\xi(z_0) - \xi(z_1))) - \mathcal{F}(z_0 - z_1)}{s} \nonumber \\
& & + \frac{\mathcal{G}(z_0 + s\xi(z_0)) - \mathcal{G}(z_0)}{s} - \frac{\mathcal{L}_1(s) - \mathcal{L}_2(s)}{s} - \frac{\mathcal{L}_2(s) - \mathcal{L}_2(0)}{s},
\eqnn
as $s \to 0$ converges to
\bqn
\nabla\mathcal{F}(z_0 - z_1) \cdot(\xi(z_0) - \xi(z_1) ) + \left( \nabla \mathcal{G}(z_0) - \nabla_{\nu} \mathcal{L}(z_0)\right)\cdot \xi(z_0).
\eqnn
Since $\nu$ has compact support, the dominated convergence theorem, the identity \r{e-intPart} and since $\nabla \mathcal{F}$ is odd, it holds
\bqn
+ \infty & > & \lim_{s \to 0} \frac{\H_c((Id + s\xi)_{\#}\nu) - \H_c(\nu)}{s} \nonumber \\
& = & \frac{1}{2} \int_{\R^{4d}} \nabla\mathcal{F}(z_0 - z_1) \cdot(\xi(z_0) - \xi(z_1) ) d\nu(z_0) d\nu(z_1) + \int_{\R^{2d}} \left( \nabla \mathcal{G}(z_0) - \nabla_{\nu} \mathcal{L}(z_0)\right)\cdot \xi(z_0) d\nu(z_0) \nonumber\\
& = & \int_{\R^{2d}} w(z_0)\cdot \xi(z_0) d\nu(z_0).
\eqnn
From the last inequality, the assumption that $|\partial \H_c|(\nu) < +\infty$ and using the trivial estimate
\bqn
\W_2((Id+s\xi)_{\#}\nu, \nu) \leq s \|\xi \|_{L^2_{\nu}},
\eqnn
we get
\bqn
\int_{\R^{2d}} w(z_0) \cdot \xi(z_0)d\nu(z_0) \leq |\partial \H_c|(\nu) \|\xi\|_{L^2_{\nu}},
\eqnn
and hence, changing the sign of $\xi$,
\bqn
\left|\int_{\R^{2d}} w(z_0)\cdot \xi(z_0) d\nu(z_0) \right| \leq |\partial \H_c|(\nu) \|\xi\|_{L^2_{\nu}}.
\eqnn
This proves that $w \in L^2_{\nu}(\R^{2d})$ and that $\|w\|_{L^2_{\nu}} \leq |\partial \H_c|(\nu)$.

We now prove that if the vector $w$ belongs to $L^2_{\nu}(\R^{2d})$, then it is in the subdifferential of $\H_c$; this shall imply, by \eqref{e-estAbNorm}, that $w \in D(|\partial \H_c|)$ and that it is a minimal selection $\partial \H_c(\nu)$, by the previous estimate and Proposition \ref{t-minnorm}.

We thus consider a test measure $\overline{\nu}$, a plan $\rho \in \Gamma_o(\nu, \overline{\nu})$, and we compute the directional derivative of $\H_c$ along the direction induced by $\rho$. Denoting with
\bqn
\mathcal{L}_1(s) & = & \ell(y,(1-s)z_0+ s z_1, \mathsmaller \int \om d((1-s)\pi_1 + s\pi_2)_{\#}\rho), \nonumber \\
\mathcal{L}_2(s) & = & \ell(y,z_0, \mathsmaller \int \om d((1-s)\pi_1 + s\pi_2)_{\#}\rho),
\eqnn
for every $s \in [0,1]$, then the map
\bqn
s &\mapsto& \frac{\mathcal{F}((1-s)(z_0 - \overline{z}_0) + s(z_1 - \overline{z}_1)) - \mathcal{F}(z_0 - \overline{z}_0)}{s} \nonumber \\
&&+ \frac{\mathcal{G}((1-s)z_0 + s z_1) - \mathcal{G}(z_0)}{s} - \frac{\mathcal{L}_1(s) - \mathcal{L}_2(s)}{s} - \frac{\mathcal{L}_2(s) - \mathcal{L}_2(0)}{s},
\eqnn
as $s \to 0$ converges to
\bqn
\nabla\mathcal{F}(z_0 - \overline{z}_0)\cdot \left((z_1 - z_0) - (\overline{z}_1 - \overline{z}_0) \right) + \left( \nabla \mathcal{G}(z_0) - \nabla_{\nu} \mathcal{L}(z_0)\right) \cdot (z_1 - z_0).
\eqnn
Hence, from Proposition \ref{l-semiconvex}, the dominated convergence theorem, the identity \r{e-intPart} and since $\nabla \mathcal{F}$ is odd, we get
\bqn
\H_c(\overline{\nu}) - \H_c(\nu) & \geq & \lim_{s \to 0} \frac{\H_c(((1-s)\pi_1 + s\pi_2)_{\#}\rho) - \H_c(\nu)}{s} + o(\W_2(\overline{\nu},\nu)) \nonumber \\
& = & \frac{1}{2} \int_{\R^{8d}}\nabla\mathcal{F}(z_0 - \overline{z}_0)\cdot \left((z_1 - z_0) - (\overline{z}_1 - \overline{z}_0) \right) d\rho(z_0,z_1)d\rho(\overline{z}_0,\overline{z}_1) \nonumber\\
& & + \int_{\R^{4d}}\left( \nabla \mathcal{G}(z_0) - \nabla_{\nu} \mathcal{L}(z_0)\right) \cdot (z_1 - z_0)  d\rho(z_0,z_1)+ o(\W_2(\overline{\nu},\nu)) \nonumber \\
& = & \int_{\R^{4d}} w(z_0)\cdot(z_1 - z_0) d\rho(z_0,z_1) + o(\W_2(\overline{\nu},\nu)).
\eqnn
We have thus proven that $w \in \partial \H_c(\nu)$.
\eproof

\section{Proof of Theorem \ref{t-main}}
\label{s-proof}

In this section, we prove Theorem \ref{t-main}. We first recall that 
we already proved in Corollary \ref{c-gamma} that there exists a mean-field optimal control for Problem \ref{problemPDE}. 
We now want to prove that all mean-field optimal controls are solutions of the extended PMP.

Let $u^*$ be a mean-field optimal control for Problem \ref{problemPDE} with initial datum $(y^0,\mu^0)$. Fix $\mu_N^0$ as in Definition \ref{d-mfoc}--\ref{uno}, and consider a sequence $(u^*_N)_{N \in \N}$ of optimal controls of Problem \ref{problemODE} with initial datum $(y^0,\mu^0_N)$, having a subsequence (which, for simplicity, we do not relabel) weakly converging to $u^*$ in $L^1([0,T];\mathcal{U})$. Denote with $(y^*_N,x^*_N)$ the trajectory of \eqref{eq:discdyn} corresponding to the control $u^*_N$ and the initial datum $(y^0,\mu^0_N)$ of Problem \ref{problemODE}. Compute the corresponding pair control-trajectory $(u^*_N,(y^*_N,q^*_N,x^*_N,p^*_N))$ satisfying the PMP for Problem \ref{problemODE}, that exists due to Theorem \ref{PMPODE}. Set $\nu^*_N:= \Pi_N(x^*_N, p^*_N)$ and $r^*_N:=Np^*_N$. By Proposition \ref{p-boundedsupp}, the trajectories $(y^*_N,q^*_N,\nu^*_N)$ are equibounded and equi-Lipschitz from $[0,T]$ to the product space $\Y=\R^{2dm}\times\PP(\R^{2d})$ endowed with the distance \
eqref{e-Y}, and the empirical measures $\nu^*_N$ have equibounded support. Moreover, the pair $(u^*_N,(y^*_N,q^*_N,\nu^*_N))$ satisfies the extended PMP by Proposition \ref{p-embed}.

By the Ascoli-Arzel\`a theorem, we have that there exists a subsequence, which we denote again with $(y^*_N,q^*_N,\nu^*_N)$, that converges to $(y^*,q^*, \nu^*):[0,T]\to \R^{dm}\times\PP(\R^{2d})$ uniformly with respect to $t \in [0,T]$. Since by definition $\pi_{1\#}\nu^*_N=\mu^*_N$, by the convergence of $\mu^*_N$ to $\mu^*$ proved in Proposition \ref{p-esistenza}, we get $\pi_{1\#}\nu^*=\mu^*$. Observe that $(y^*,q^*, \nu^*)$ is a Lipschitz function with respect to time and $\nu^*$ has support contained in $B(0,R_T)$ for all $t\in[0,T]$. Moreover, by the boundary conditions for each $N$, we have that $y^*(0)=y^0$, $\pi_{1\#}(\nu^*(0))=\mu^0$ and $q^*(T)=0$, $\pi_{2\#}(\nu^*(T))(r)=\delta(r)$.

Fix now $t\in [0,T]$. To shorten notation, let $E\colon \R^{dm}\times \R^D\to \R$ be the functional, strictly concave with respect to $u$, defined as
$$
E(q,u)= \sum_{k=1}^m q_k\cdot B_k u -\gamma(u)\,.
$$
Recall that by \eqref{eq:PMPdiscrete} and by Remark \ref{r-u}, $u^*_N(t)$ satisfies
$$
u^*_N(t)=\arg\max_{u\in \mathcal U} E(q^*_N(t), u)\,,
$$
since the maximum is uniquely determined by strict concavity. Since $\mathcal U$ is bounded, by definition $E(\cdot, u)$ is continuous uniformly with respect to $u \in \mathcal U$. The convergence of $q^*_N(t)$ to $q^*(t)$ then implies that every accumulation point $v_t\in \mathcal U$ of $u^*_N(t)$ must satisfy 
\begin{align}\label{argmax}
v_t=\arg\max_{u\in \mathcal U} E(q^*(t), u)
\end{align}
and is therefore uniquely determined. This shows that the sequence $u^*_N$ is pointwise converging in $[0,T]$ to the function $v(t):=v_t$. Due to the boundedness of $\mathcal U$, we further have that $u^*_N \to v$ in $L^1((0,T); \mathcal U)$. Since $u^*_N$ was already converging to $u^*$ weakly in $L^1((0,T); \mathcal U)$ it must be $u^*(t)= v(t)$ for a.e.\ $t\in (0,T)$, which together with \eqref{argmax} implies that
\begin{align}\label{L1forte}
u^*_N \to u^* \hbox{ strongly in }L^1((0,T); \mathcal U) 
\end{align}
and that
$$
u^*(t)=\arg\max_{u\in \mathcal U} E(q^*(t), u)
$$
for a.e.\ $t\in [0,T]$. Due to the explicit expression of $\H(y,q,\nu, u)$ in \eqref{e-H}, this is equivalent to say that
$$
\H(y^*(t),q^*(t),\nu^*(t),u^*(t))=\arg \max_{u\in \mathcal U}\H(y^*(t),q^*(t),\nu^*(t),u)
$$
for a.e.\ $t\in [0,T]$.


We finally prove that $(y^*,q^*,\nu^*)$ satisfies the Hamiltonian system \r{e-PMP} with control $u^*$. Due to equi-Lipschitz continuity, we have that the derivatives  $(\dot y^*_N,\dot q^*_N)$, and $\partial_t\nu^*_N$ converge to $(\dot y^*,\dot q^*)$, and $\partial_t \nu^*$, respectively, weakly in $L^1([0,T];\R^{2md})$ and in the sense of distributions.
Observe now that by \eqref{e-gradx} and \eqref{e-gradr} the vector field $\nabla_{\nu} \H_c(y,q,\nu)(\cdot,\cdot)$, which is independent of $u$, is continuously depending on $(y,q,\nu)$. By the uniform convergence of $(y^*_N,q^*_N,\nu^*_N)$ and since ${\rm supp}(\nu^*_N(t))\subset B(0, R_T)$ for all $t\in [0,T]$ we get that
$$
\nabla_{\nu} \H_c(y^*_N(t),q^*_N(t),\nu^*_N(t))(x,r)\rightrightarrows \nabla_{\nu} \H_c(y^*(t),q^*(t),\nu^*(t))(x,r)
$$
uniformly with respect to $t\in [0,T]$ and $(x,r)\in B(0, R_T)$. From this, using again the narrow convergence of $\nu^*_N(t)$ to $\nu^*(t)$ and since ${\rm supp}(\nu^*_N(t))\subset B(0, R_T)$, we then get the uniform bound
$$
\|\left(J\nabla_{\nu} \H_c(y^*_N(t),q^*_N(t),\nu^*_N(t))\right)\nu^*_N(t)\|_{M_b(\R^D, \R^D)}\le C_T,
$$
for some constant $C_T$ independent of $t\in [0,T]$, as well as the narrow convergence
$$
\left(J\nabla_{\nu} \H_c(y^*_N(t),q^*_N(t),\nu^*_N(t))\right)\nu^*_N(t)\rightharpoonup\left(J\nabla_{\nu} \H_c(y^*(t),q^*(t),\nu^*(t))\right)\nu^*(t)
$$
for all $t\in [0,T]$. Testing with functions $\phi \in \mathcal{C}^\infty_c([0,T]\times \R^{2d};\R)$, the two above properties are enough to show that
$$
\nabla_{(x,r)}\cdot \Pt{(J\nabla_{\nu}\H_c(y^*_N(t),q^*_N(t),\nu^*_N(t)))\nu^*_N(t)}\rightharpoonup\nabla_{(x,r)}\cdot \Pt{(J\nabla_{\nu}\H_c(y^*(t),q^*(t),\nu^*(t)))\nu^*(t)}
$$
in the sense of distributions, so that $\nu^*$ solves the third equation in \eqref{e-PMP}.

For all $k=1,\dots, m$,  taking derivatives in the explicit expression in \eqref{e-H} and using the definition of $\H_c$, we have that $\nabla_{y_k}\H_c(y,q,\nu, u)$ is actually independent of $u$ and is continuous with respect to the Euclidean convergence on $(y,q)$ and the narrow convergence on measures $\nu$ with compact support in a fixed ball $B(0, R_T)$. Therefore, since $(y^*_N,q^*_N,\nu^*_N)$ converges to $(y^*,q^*, \nu^*)$ uniformly with respect to $t \in [0,T]$, and there is no dependence on $u$, for all $k=1,\dots, m$ we have that
$$
\nabla_{y_k}\H_c(y^*_N(t),q^*_N(t),\nu^*_N(t), u^*_N(t))\to \nabla_{y_k}\H_c(y^*(t),q^*(t),\nu^*(t), u^*(t))
$$
in $\R^d$ uniformly with respect to $t \in [0,T]$. It then follows that $q^*$ solves the second equation in \eqref{e-PMP}.

A similar argument, also using the $L^1$ convergence of $u^*_N$ to $u^*$ proved in \eqref{L1forte}, shows that
$$
\nabla_{q_k}\H_c(y^*_N(t),q^*_N(t),\nu^*_N(t), u^*_N(t))\to \nabla_{q_k}\H_c(y^*(t),q^*(t),\nu^*(t), u^*(t))
$$
in $L^1([0,T];\R^d)$ for all $k=1,\dots, m$, so that $y^*$ solves the first equation in \eqref{e-PMP}. This concludes the proof of Theorem \ref{t-main}.

\section{An example}\label{s-CS}
In this section, we show the application of the extended Pontryagin Maximum Principle to a toy model for crowd interactions. The Cucker-Smale model, introduced in \cite{CS}, was first studied in its mean-field limit form in \cite{ha2008from}. It models the phenomenon of alignment of velocities in crowds, that can be observed, e.g., in flocks of birds.

In this model, each agent is identified by its position $x_i$ and velocity $v_i$, and it adjusts its velocity by relaxing it towards a weighted mean of the velocities of the group. The weight is a nonincreasing function $\phi$ of the distance between individuals. In the original paper \cite{CS}, the authors propose $\phi(\lambda)=\frac{K}{(\sigma^2+\lambda^2)^\beta}$, for some fixed parameters $K,\sigma>0$ and $\beta\geq 0$. For our computations, we consider $\phi\in \mathcal{C}^2(\R^d,\R^+)$ being a radial function.

The finite-dimensional dynamics is given by the ODE system
\begin{align*}
\begin{cases}
\begin{aligned}
\dot x_i&=v_i, \\
\dot v_i&=\frac{1}{N}\sum_{j = 1}^N \phi(x_i-x_j) (v_j-v_i),
\end{aligned}\quad \quad i = 1, \ldots, N.
\end{cases}
\end{align*}
We add to it $m$ leaders with positions and velocities given by $(y_k,w_k)$ for every $k = 1, \ldots, m$, on which a control variable $u_k$ is active. Since the control acts as an external force, $u_k$ will directly affect the evolution of the velocities $w_k$ only. The mean-field limit for $N\to+\infty$ of the resulting system is given by (see, e.g., \cite{fornasier2014mean})
\begin{align}\label{e-CS}
\begin{cases}
\begin{aligned}
\dot y_k&=w_k,\\
\dot w_k&=(\Phi\star\mu)(y_k,w_k)+\frac{1}{m}\sum_{j = 1}^m \phi(y_k-y_j) (w_j-w_k)+u_k,\\
\end{aligned} \quad \quad \quad k=1,\ldots, m\\
\begin{aligned}
\partial_t \mu&=- v\cdot\nabla_x\mu-\nabla_v \cdot \left[\left(\Phi\star\mu+\frac{1}{m}\sum_{j = 1}^m \phi(x-y_j) (w_j-v)\right)\mu\right],
\end{aligned} 
\end{cases}
\end{align}
where $\mu=\mu(x,v)$ is the density of followers and $\Phi(x,v):=\phi(x)(-v)$. Notice, this is a particular case of \r{eq:mfdyn}, where the state variables for the leaders are ${\bf y}_k:=\Pt{\ba{c}y_k\\w_k\ea}$ and ${\bf y}=\Pt{{\bf y}_1,\ldots, {\bf y}_m}$, the ones for the followers are ${\bf x}=\Pt{\ba{c} x\\v\ea}$, and one chooses
\begin{align*}
K({\bf x}):=\Pt{\ba{c}0\\ \Phi({\bf x})\ea}, &\quad f_k({\bf y})=\Pt{\ba{c}w_k\\ \frac{1}{m}\sum_{j = 1}^m \phi(y_k-y_j) (w_j-w_k)\ea},\\
g({\bf y})({\bf x})=&\Pt{\ba{c}v\\ \frac{1}{m}\sum_{j = 1}^m \phi(x-y_j) (w_j-v)\ea},
\end{align*}
and, for every $k = 1, \ldots, m$, $B_k$ is the $2d \times (dm)$ matrix that maps $u=\Pt{\ba{c}u_1\\ \dots \\u_m\ea} \in \R^{dm}$ into the element $\Pt{\ba{c}\textbf{0}\\u_k\ea}\in \R^{2d}$. Notice that since $\phi$ is a radial function, the function $\Phi$, and thus $K$, is odd.

A standard problem in the study of the Cucker-Smale model is to find conditions to ensure flocking, i.e., alignment of the whole crowd to the same velocity. For this reason, it is interesting in our case to study the minimization of the variance\footnote{For simplicity of computation, we consider minimization of 4 times the variance.} of the crowd, by choosing
\begin{align}
\begin{aligned}\label{e-CScost}
L({\bf y},\mu)&:=\int_{\R^{2d}}\Pt{ \frac{2}{m}\sum^m_{k = 1} \|w_k\|^2+2\|v\|^2}\,d\mu(x,v)-\left\|\frac{1}{m}\sum^m_{k = 1} w_k+ \int v\,d\mu(x,v)\right\|^2\\
&=\int_{\R^{2d}} \Pt{\frac{2}{m}\sum^m_{k = 1} \|w_k\|^2+2 \|v\|^2-\Pt{\frac{1}{m}\sum^m_{k = 1} w_k+ \int v'\,d\mu(x',v')}\cdot\Pt{ \frac{1}{m}\sum^m_{k = 1} w_k+ v}}\,d\mu(x,v),
\end{aligned}
\end{align}
that is of the form $L=\int_{\R^{2d}} \ell({\bf y},{\bf x},\int \omega\mu)\,d\mu({\bf x})$ by choosing $\omega({\bf x})=v$ and
\bqn
\ell({\bf y},{\bf x},\varsigma)=\frac{2}{m}\sum^m_{k = 1} \|w_k\|^2+2 \|v\|^2-\Pt{\frac{1}{m}\sum^m_{k = 1} w_k+ \varsigma}\cdot\Pt{ \frac{1}{m}\sum^m_{k = 1} w_k+ v}.
\eqnn
For the control constraints, we assume $\mathcal U:=[-1,1]^{dm}$ and we choose to penalize the $L^2$-norm of the control, hence $\gamma(u):=\|u\|^2$.
\brem Other forms for the cost $L$ can be of interest. For example, on may want to drive the crowd to a {\it given} fixed velocity $\bar v$. In this case, one can minimize
\bqn
L_1({\bf y},\mu)&:=&\int_{\R^{2d}}\Pt{ \frac{1}{2m}\sum^m_{k = 1} \|w_k-\bar v\|^2+\frac12 \|v-\bar v\|^2}\,d\mu(x,v),
\eqnn
that is again of the form $\int_{\R^{2d}} \ell({\bf y},{\bf x},\int \omega\mu)\,d\mu({\bf x})$, with $\ell$ not depending on its third variable, this time.
\erem
Since Hypotheses (H) are clearly satisfied, we now apply the extended Pontryagin Maximum Principle to the optimal control problem with cost functional \r{e-CScost} constrained by the system \r{e-CS}. For simplicity of notation, we study the 1-dimensional problem, i.e., $d=1$. We introduce the dual variables of ${\bf y}_k$ and ${\bf x}$ denoted by ${\bf q}_k=(q_k,z_k)$ and ${\bf r}=(r,s)$, respectively. The Hamiltonian $\H$ in \r{e-H} can be found by direct substitution:
\newcommand{\dnu}{\,d\nu(x,v,r,s)}
\newcommand{\dnup}{\,d\nu(x',v',r',s')}
\bqn
\H({\bf y},{\bf q},\nu, u)&=&\frac12\int_{\R^{8}}(s-s')\phi(x-x')(v'-v)\dnup\dnu\nn
&&+\int_{\R^{4}}\left(rv+s\frac{1}{m}\sum_{j = 1}^m \phi(x-y_j) (w_j-v)\right)\dnu\nn
&&+\sum_{k=1}^m \Pt{q_kw_k + z_k\int_{\R^{4}} \phi(y_k-x)(v-w_k)\dnu}\nn
&&+\sum_{k=1}^m \Pt{z_k\frac{1}{m}\sum_{j = 1}^m \phi(y_k-y_j) (w_j-w_k)+z_k u_k}\nn
&&-\int_{\R^{2}}\ell({\bf y},{\bf x},\mbox{$\int\omega\mu$})\,d\mu({\bf x})-|u|^2.
\eqnn
The optimal control can be explicitly computed by \r{e-u} as follows
\bqn
u^*_k({\bf y},{\bf q})=\begin{cases}
\frac{1}{2}z_k&\mbox{~~if } z_k\in[-2,2],\\
{\rm sign}(z_k)&\mbox{~~otherwise}.
\end{cases}
\eqnn
Denoting with $\mu$ the first marginal of $\nu$, the PMP dynamics of the state and adjoint variables is given by
\begin{align*}
\begin{cases}
\dot y_k=&w_k,\\
\dot w_k=&(\Phi\star\mu)(y_k,w_k)+\frac{1}{m}\sum_{k = 1}^m \phi(y_k-y_j) (v_j-v_k)+u^*_k({\bf y},{\bf q}),\\
\dot q_k=&\frac1m \int_{\R^4} s \phi'(x-y_k)(w_k-v)\dnu-z_k\int_{\R^4} \phi'(y_k-x)(v-w_k)\dnu\\
&-\frac{1}{m}\sum_{j\neq k} z_j\phi' (y_k-y_j)(w_j-w_k),\\
\dot z_k=& -\int_{\R^4} \Pt{\frac1m s-z_k}\phi(x-y_k)\dnu-q_k+\sum_{j\neq k} z_j\phi(y_k-y_j)\\
&-\frac{4}{m} w_k+\frac{2}{m^2} \sum_{j = 1}^m w_j +\frac{2}{m}\int_{\R^2} v\,d\mu(x,v),\\
\partial_t\nu=&-\nabla_{(x,v,r,s)}\cdot\Pt{(J\nabla_\nu \H_c({\bf y},{\bf q},\nu, u^*))\nu},
\end{cases}
\end{align*}
where the components of the vector field $\nabla_\nu \H_c({\bf y},{\bf q},\nu, u^*)$ are given at every point $(x,v,r,s) \in \R^4$ by
\begin{align*}
\nabla_\nu \H_c\cdot e_1=&\int_{\R^4}(s-s')\phi'(x-x')(v'-v)\dnup+s\frac{1}{m}\sum_{j = 1}^m \phi'(x-y_j) (w_j-v)\\
&-\sum_{k=1}^m z_k \phi'(y_k-x)(v-w_k),\\
\nabla_\nu \H_c\cdot e_2=&-\int_{\R^4}(s-s')\phi(x-x')\dnup+r-s\frac{1}{m}\sum_{j = 1}^m \phi(x-y_j)+\sum_{k=1}^m z_k\phi(y_k-x)-3v\\
&+\frac{2}{m}\sum^m_{k =1}w_k+\int_{\R^2} v d \mu(x,v),\\
\nabla_\nu \H_c\cdot e_3=&\;v,\\
\nabla_\nu \H_c\cdot e_4=&\;(\Phi\star\mu)(x,v)+\frac{1}{m}\sum_{j = 1}^m \phi(x-y_j) (w_j-v).
\end{align*}

We remark that, as it happens for the standard PMP, the explicit computation of the third and the fourth components gives exactly the vector field determining the dynamics of $\mu$, the first marginal of $\nu$, in accordance with \eqref{e-CS}.

\section*{Acknowledgement}

The authors acknowledge the support of the PHC-PROCOPE Project ``Sparse Control of Multiscale Models of Collective Motion''. Mattia Bongini and Massimo Fornasier additionally acknowledge the support of the ERC-Starting Grant Project ``High-Dimensional Sparse Optimal Control''.

\bibliographystyle{abbrv}
\bibliography{bibliopmp}

\end{document}